\documentclass[letterpaper,reqno,11pt,oneside]{amsart}

\usepackage{amsmath,amsthm,amsfonts,amssymb}
\usepackage{enumitem,comment}
\usepackage[usenames,dvipsnames]{color}
\usepackage{mathtools}
\usepackage{mathrsfs}
\usepackage[extension=pdf]{hyperref}
\usepackage[totalheight=9.6in,totalwidth=5.95in,centering]{geometry}
\usepackage{graphics,graphicx}
\usepackage[color,notref,notcite,final]{showkeys}
\usepackage[style=alphabetic,natbib=true,maxnames=99,isbn=false,doi=false,url=false,firstinits=true,hyperref=auto,arxiv=abs,backend=bibtex]{biblatex}
\AtEveryBibitem{%
  \clearlist{language}}
\DeclareFieldFormat[article,inbook,incollection,inproceedings,patent,thesis,unpublished]{title}{#1\isdot}
\renewbibmacro{in:}{%
  \ifentrytype{article}{}{%
    \printtext{\bibstring{in}\intitlepunct}}} 
\defbibheading{apa}[\refname]{\section*{#1}}
\DeclareFieldFormat{sentencecase}{\MakeSentenceCase{#1}} \renewbibmacro*{title}{%
  \ifthenelse{\iffieldundef{title}\AND\iffieldundef{subtitle}}{}
  {\ifthenelse{\ifentrytype{article}\OR\ifentrytype{inbook}%
      \OR\ifentrytype{incollection}\OR\ifentrytype{inproceedings}%
      \OR\ifentrytype{inreference}} {\printtext[title]{%
        \printfield[sentencecase]{title}%
        \setunit{\subtitlepunct}%
        \printfield[sentencecase]{subtitle}}}%
    {\printtext[title]{%
        \printfield[titlecase]{title}%
        \setunit{\subtitlepunct}%
        \printfield[titlecase]{subtitle}}}%
    \newunit}%
  \printfield{titleaddon}}

\definecolor{darkblue}{rgb}{0.13,0.13,0.39}
\hypersetup{colorlinks=true,urlcolor=darkblue,citecolor=darkblue,linkcolor=darkblue,%
  pdftitle=How flat is flat in random interface growth?}

\mathtoolsset{showmanualtags,showonlyrefs}

\newtheorem{thm}{Theorem}[section] 
\newtheorem{lem}[thm]{Lemma}
 
\newtheorem{prop}[thm]{Proposition} \newtheorem{cor}[thm]{Corollary}
 
\theoremstyle{definition} \newtheorem{rem}[thm]{Remark} \newtheorem*{rem*}{Remark}
\newtheorem*{ldp}{Large deviation principle}
\newtheorem{defn}[thm]{Definition} \newtheorem{ex}[thm]{Example}
 \newcounter{assum}

\newcommand{\I}{{\rm i}} \newcommand{\pp}{\mathbb{P}} 
  \newcommand{\rr}{\mathbb{R}}
\newcommand{\nn}{\mathbb{N}}  \newcommand{\aip}{\mathcal{A}}
 
\newcommand{\aipt}{\mathcal{A}^t}
\newcommand{\tilaip}{\wt{\mathcal{A}}}

\newcommand{\p}{\partial}
\newcommand{\uno}[1]{\mathbf{1}_{#1}}
\newcommand{\ep}{\varepsilon}

\newcommand{\wt}{\widetilde}

\newcommand{\K}{K_{\Ai}}  
\newcommand{\qand}{\quad\text{and}\quad}
\newcommand{\qqand}{\qquad\text{and}\qquad}

\DeclareMathOperator{\Ai}{Ai}

\newcommand{\ts}{\hspace{0.1em}}
\newcommand{\tts}{\hspace{0.05em}}
\newcommand{\tsm}{\hspace{-0.1em}}

\newcommand{\inv}[1]{\frac{1}{#1}}

\newcommand{\mfH}{\mathfrak{H}}

\newcommand{\hnw}{h^{\rm nw}}
\newcommand{\mfHnw}{\mfH^{\rm nw}}
\newcommand{\scat}{\mathcal{S}}

\def\dash---{\kern.16667em---\penalty\exhyphenpenalty\hskip.16667em\relax}

\numberwithin{equation}{section}

\let\oldmarginpar\marginpar
\renewcommand\marginpar[1]{\-\oldmarginpar[\raggedleft\footnotesize #1]%
  {\raggedright{\small\textsf{#1}}}}

\allowdisplaybreaks[2]

\begin{document}

\title{How flat is flat in random interface growth?}

\author{Jeremy Quastel} \address[J.~Quastel]{
  Department of Mathematics\\
  University of Toronto\\
  40 St. George Street\\
  Toronto, Ontario\\
  Canada M5S 2E4} \email{quastel@math.toronto.edu}

\author{Daniel Remenik} \address[D.~Remenik]{
  Departamento de Ingenier\'ia Matem\'atica and Centro de Modelamiento Matem\'atico\\
  Universidad de Chile\\
  Av. Beauchef 851, Torre Norte, Piso 5\\
  Santiago\\
  Chile} \email{dremenik@dim.uchile.cl}

\begin{abstract}
Domains of attraction are identified for the universality classes of one-point asymptotic fluctuations for the Kardar-Parisi-Zhang (KPZ) equation with general initial data.
The criterion is based on a large deviation rate function for the rescaled initial data, which  arises naturally from the Hopf-Cole transformation.
This allows us, in particular, to distinguish the domains of attraction of \emph{curved}, \emph{flat}, and \emph{Brownian} initial data, and to identify the boundary between the curved and flat domains of attraction, which turns out to correspond to square root initial data.

The distribution of the asymptotic one-point fluctuations is characterized by means of a variational formula written in terms of certain limiting processes (arising as subsequential limits of the spatial fluctuations of KPZ equation with narrow wedge initial data, as shown in \cite{corwinHammondKPZ}) which are widely believed to coincide with the Airy$_2$ process.
In order to identify these distributions for general initial data, we extend earlier results on continuum statistics of the Airy$_2$ process to probabilities involving the process on the entire line.
In particular, this allows us to write an explicit Fredholm determinant formula for the case of square root initial data.
\end{abstract}

\maketitle

\section{Introduction and main results}
\label{sec:intro}

\subsection{Motivation}

The last five years have seen the discovery of some special exact solutions of the \emph{Kardar-Parisi-Zhang (KPZ) equation} \cite{kpz},
\begin{equation}\label{eq:KPZ}
    \partial_t h = \tfrac12 \partial_x^2 h - \tfrac12(\partial_x h)^2 + \xi\,.
\end{equation}
Here $\xi$ is space-time Gaussian white noise.
The exact solutions are for very special choices of initial conditions, and
the long exact computations leading to them break down under the smallest perturbations.  In the long time limit,
these fluctuations converge in distribution to laws, many of them coming from random matrices, that are supposed to
be universal throughout the KPZ universality class, but which do depend on the initial condition type.
  In a recent talk by one of 
us, while we were fumbling to prove such a result for the very natural (but again, very special) \emph{flat} initial condition $h_0(x)=0$ \cite{oqr-flat}, J\"urg Fr\"ohlich asked if we couldn't say \emph{anything} about what happens for more general initial conditions.  In this article, we provide some partial answers.

The KPZ equation \eqref{eq:KPZ} models a large class of randomly growing one-dimensional interfaces, directed random polymers and interacting particle systems.
It is expected to arise, in particular, as a scaling limit of a wide range of related systems (for some of which rigorous results are available \cite{berGiaco,acq,akq2,mqrScaling,corwinTsaiHSEP}).
See the reviews \cite{corwinReview,quastelCDM,quastelRem-review,quastelSpohn} for more background on the KPZ equation and its universality class.

Recall that the \emph{Hopf-Cole solution} of \eqref{eq:KPZ} is defined as
\begin{equation}
  \label{eq:hopf-cole}
  h(t,x)=-\log Z(t,x)
\end{equation}
where $Z(t,x)$ solves the \emph{stochastic heat equation (SHE) with multiplicative noise}
\begin{equation}
  \label{eq:SHE}
  \p_tZ=\tfrac12\p_x^2Z-\xi Z.
\end{equation}
It has been understood for almost two decades now \cite{berGiaco} that this is the physically relevant notion of solution of \eqref{eq:KPZ}, which, since solutions are locally Brownian in space,  is ill-posed as written.
Recently several well-posedness results have become available  \cite{hairer,hairerReg, Gubinelli:2015aa} which 
coincide with the Hopf-Cole solutions.
In everything that follows we will take \eqref{eq:hopf-cole} as the \emph{definition} of the solution of \eqref{eq:KPZ}.

The first result on exact solutions of the KPZ equation came in 2010, when a number of independent groups \cite{acq,sasamSpohn,cal-led-rosso,dotsenkoGUE} succeeded in computing the one-point distribution
\begin{equation}\label{eq:defFtx}
F_{t,x}(r) = \pp\!\left(\frac{-h(t,2^{1/3}t^{2/3}x)+\inv{24}t}{2^{-1/3}t^{1/3}}\leq r\right)
\end{equation}
for the special \emph{narrow wedge} (sometimes also referred to as \emph{curved}) initial data given (in terms of \eqref{eq:hopf-cole})\footnote{It is most natural, as in \eqref{nw}, to state initial data for KPZ in terms of that of the stochastic heat equation, then to take log for $t>0$.\label{foot:initSHE}} by
\begin{equation}\label{nw}Z(0,x) =\delta_0(x).\end{equation}
In this case, it is not hard to see that $F_{t,x}$ only depends on $x$ through a parabolic shift.
The exact formula is a complicated integral of Fredholm determinants, but it is simple enough to see from it that (see \cite[Cor. 1.6]{acq})
\begin{equation}\label{curvedass}
\lim_{t\to\infty}F_{t,x}(r) =F_{\rm GUE}\big(r+x^2\big),
\end{equation}
which is also given by a Fredholm determinant.
Here $F_{\rm GUE}$ denotes the Tracy-Widom GUE distribution \cite{tracyWidom}, which was first discovered in the context of random matrices, where it governs the large $N$ asymptotics of the top eigenvalue of an $N\times N$ matrix drawn from the \emph{Gaussian Unitary Ensemble}.
More precisely, if $\mathcal{N}(0,\sigma^2)$ denotes a centered Gaussian random variable with variance $\sigma^2$ and we  let $a_{ij} = \overline{a_{ji}} = \mathcal{N}(0,N/2)+\I\hspace{0.1em}\mathcal{N}(0,N/2)$ for $i<j=1,\ldots N$ and $a_{ii}=\mathcal{N}(0,N)$, with all the Gaussian variables assumed independent, then for the largest eigenvalue $\lambda_N$ of the matrix $(a_{ij})_{i,j=1}^N$ we have
\begin{equation}\label{GUE}
 \pp\!\left(\frac{\lambda_N -2N}{N^{1/3}} \leq r\right) \xrightarrow[N\to \infty]{} F_{\rm GUE}(r).
\end{equation}
Its appearance in random growth models was a surprise, and goes back to the breakthrough papers \cite{baikDeiftJohansson,johansson}.
To a large extent the connection has been explained through the umbrella of Macdonald processes \cite{borCor}.

Another basic initial data for the KPZ equation is \emph{flat}, say $Z(0,x)=1$.
This case is far less understood.
At the present time there is only a conjectural Pfaffian formula \cite{cal-led} derived via Bethe ansatz, divergent series, and a number of uncontrolled approximations (see also \cite{oqr-flat} for partial progress on rigorous results).
From the formula, a formal computation shows that in this case
\begin{equation}\label{flatass}
\lim_{t\to\infty}F_{t,x}(r) = F_{\rm GOE}(4^{1/3}r),
\end{equation}
which is described by a Fredholm Pfaffian (a rigorous proof of this identity, as well as a rigorous formula for $F_{t,x}$, remains a considerable technical challenge).
The Tracy-Widom GOE distribution $F_{\rm GOE}$ \cite{tracyWidom2} which appears in \eqref{flatass} is the analogue of \eqref{GUE} for real symmetric Gaussian matrices, which form the \emph{Gaussian Orthogonal Ensemble} (see e.g. \cite[Sec. 1.2.2]{quastelRem-review} for the explicit definition).
A partial explanation of the connection between GOE and random growth models, which is not nearly as well understood as in the GUE/curved case, is provided in the recent paper \cite{nguyenRemenik}.

Whether or not \eqref{flatass} can be proved, there is little doubt that it is
true.  Moreover, the common wisdom is that \eqref{curvedass} holds whenever the initial
data is \emph{curved} and \eqref{flatass} holds whenever the initial data is \emph{flat}.
The question is, of course: \emph{how curved is curved, and how flat is flat?} 
We were very surprised both that a fairly precise answer is available, and by the answer itself (see Example \ref{1} below). 

More generally, we will consider the problem of characterizing the distribution of the limiting fluctuations of $h(t,x)$ (as in \eqref{curvedass}) for broader families of initial data. In particular, we will be interested in understanding which classes of initial data share the same limiting fluctuations.
The context in which a (partial) answer to these questions is available is a remarkable recent result of
\citet{corwinHammondKPZ}. In fact, some of our results (particularly Theorem \ref{thm:lapl-KPZ}) are to some extent implicitly contained in that article, and our purpose here is to make them explicit.

The distributions of the asymptotic one-point fluctuations in the KPZ class are characterized by means of variational formulas involving the Airy process.   However, one does not know at this time that the narrow wedge solution of the KPZ equation converges to the Airy process.  From
\cite{corwinHammondKPZ} one does have the existence of subsequential limits of the spatial fluctuations of KPZ equation with narrow wedge initial data, which turns out to be sufficient to identify domains of attraction.  Following this, we  assume the variational formula for the  Airy$_2$ process, and extend earlier results on its continuum statistics to the entire line in order to identify these distributions for general initial data.
In particular, this allows us to write an explicit Fredholm determinant formula for the case of square root initial data.

The problem of identifying limiting one-point distributions for general initial data was  considered independently in the context of  geometric last passage percolation/totally asymmetric exclusion process in \cite{corwinLiuWang} and for spatially homogeneous initial data in \cite{chhitaFerrariSpohn} in the exponential case.  In both cases, one does have the convergence to the Airy process, so the general results about convergence of one-point distributions are more complete (see Remark \ref{rem:CLW}).

\subsection{Criterion for the asymptotic fluctuations}
\label{sub:criterion}

In order to understand the dependence on initial data, we extend \eqref{curvedass} to
space.  Define the \emph{crossover Airy process} $\mathcal{A}^t$ in terms of the
\emph{narrow wedge solution} $\hnw$ (i.e. with initial data \eqref{nw}) by
\begin{equation}
  \label{eq:At}
  \mathcal{A}^t(2^{-1/3}x) = \frac{-\hnw(t,t^{2/3}x)+\inv{24}t+\inv{2}t^{1/3}x^2}  
      {2^{-1/3}t^{1/3}}.
\end{equation}
One of the main conjectures in the field \cite{acq,prolhacSpohn}\footnote{The precise scaling comes from \cite[(1.17)]{prolhacSpohn}.} is that, as $t\to\infty$,
\begin{equation}\label{airylim}
\mathcal{A}^t(x) \longrightarrow \aip(x)
\end{equation}
where $\aip(x)$ is the Airy$_2$ process, which is a \emph{universal}\footnote{By universal here it is meant that the same process appears in the asymptotics of a wide variety of models.
The analogue of \eqref{airylim} is know for several discrete models with determinantal structure, see e.g. \cite{johanssonRMandDetPr}.} process governing the spatial fluctuations of models in the KPZ universality class with curved initial data; see e.g. \cite{quastelCDM,quastelRem-review} for more on this and also \eqref{eq:airy2PrDef} below for the definition of the Airy$_2$ process $\aip$.
What is known at the present time is 

\begin{thm}[{\cite[Thm. 2.14]{corwinHammondKPZ}}]\label{thm:CH}
The laws of the $\mathcal{A}^t(\cdot)$ are tight as probability measures $\mathscr{P}_t$ on $\mathcal{C}(\rr)$ endowed with the topology of uniform convergence on compact sets.
\end{thm}

In what follows we will use the notation $\tilaip$ for the process whose law is given by any subsequential limit $\wt{\mathscr{P}}$ of the family of laws $\big(\mathscr{P}_t\big)_{t>0}$.
While $\tilaip$ is clearly the Airy$_2$ process, we presently lack the tools to rigorously identify it as such (note, however, that by \eqref{curvedass} we do know that $\tilaip$ has the correct Tracy-Widom GUE one-point marginals).

Within the KPZ universality class, the KPZ equation has the special property that it is
exponentially linear in the initial data, i.e. for any reasonable
 non-negative Borel measure $\mu$ on
$\rr$,
\begin{equation}\label{eq:linearSHE}
  h(t,x)=-\log\int_{\rr}e^{-\hnw(t,x-y)}\,d\mu(y)=-\log \int_{\rr}e^{-\frac{(x-y)^2}{2t} -\frac1{24}t+2^{-1/3}t^{1/3}\mathcal{A}^t(2^{-1/3}t^{-2/3}(x-y))}\,d\mu(y)
\end{equation}
is the solution of KPZ with initial data $Z(0,dx) = d\mu(x)$. The measures $\mu$ may be random (but we assume they are independent of the white noise) and need not have finite mass. By $\mu$ being reasonable we will mean that, for some $\kappa>0$,
\begin{equation}\label{eq:muCond}
   \|\mu\|_\kappa:=\sup_{|f(x)|\le e^{-\kappa|x|} } \int\!f\,d\mu \leq M<\infty
 \end{equation}
almost surely for some $M$ (see \cite[Section 2]{mqrScaling}).  We will denote the class of such measures by $\mathscr{M}$.

\begin{ex}\label{bmex}
One of the most natural initial conditions is $d\mu(y) =e^{B(y)} dy$ where $B(y)$ is a  two-sided Brownian motion.
Modulo a (non-trivial) height shift, it is the equilibrium distribution for \eqref{eq:KPZ}.
\end{ex}

We rescale \eqref{eq:linearSHE} to obtain
\begin{equation}\label{eq:At2}
  \frac{-h(t,2^{1/3}t^{2/3}x)+\inv{24}t}{2^{-1/3}t^{1/3}} 
  =\frac1{2^{-1/3}t^{1/3}}  \log \int e^{2^{-1/3}t^{1/3}\big[\mathcal{A}^t(x-y)-(x-y)^2\big] }\, d\mu_t(y)
\end{equation}
where, for Borel sets $\mathcal{B}$, 
\begin{equation}\label{eq:mut}
  \mu_t(\mathcal{B}) = \mu(2^{1/3}t^{2/3} \mathcal{B}).
\end{equation}
The right hand side of \eqref{eq:At2} suggests that, in order to study the limiting fluctuations of $h$ with general initial data, it should be possible to use a random version of Laplace's method.
The hypotheses of such a result are naturally stated in the language of large deviation theory.
In order to state it, we need to introduce some general definitions and notation.

Let $\{\mathsf{Q}_t\}_{t>0}$ denote a family of probability measures on  $\mathscr{M}$, and let $\nu_t$ be a random measure in $\mathscr{M}$ with law $\mathsf{Q}_t$.
For example, in our application to the KPZ equation $\nu_t$ will be taken to be the measure $\mu_t$ defined in \eqref{eq:mut}.
We will assume that  $\{\nu_t\}_{t>0}$ satisfies the following

\begin{ldp}
  There exists a (possibly random) lower semi-continuous function $I\!:\rr
  \longrightarrow (-\infty,\infty]$  (which is not identically $\infty$) and a (deterministic) increasing function $a_t\nearrow \infty$ such that
  \begin{equation}\label{eq:ldp}
    \begin{aligned}
     \limsup_{t\to \infty}a_t^{-1} \log \nu_t(\mathcal{C})  &\le -\inf_{y\in \mathcal{C}} I(y)
     \quad\text{for all closed }\mathcal{C}\subseteq\rr\qquad\text{and}\\
     \liminf_{t\to \infty}a_t^{-1} \log \nu_t(\mathcal{O})  &\ge -\inf_{y\in \mathcal{O}} I(y)
     \quad\text{for all open }\mathcal{O}\subseteq\rr
   \end{aligned}
 \end{equation}
  in distribution. When this holds, we say that the family of measures $\{\nu_t\}_{t>0}$ satisfies a \emph{large deviation principle (in distribution) with rate function $I$ and speed $a_t$} (for simplicity we will usually omit writing \emph{in distribution}). 
\end{ldp}

In this definition, given a sequence of random variables $(X_n)_{n\geq1}$ and another random variable $X$,
we say that $\limsup_{n\to\infty} X_n \le X$ \emph{in distribution} if for all $r\in \rr$,
\begin{equation}
  \limsup_{n\to \infty} \pp( X_n \ge r) \le \pp( X\ge r).
\end{equation}
Note that the probability spaces on which the random variables are defined need have nothing to do with each other, and are not relevant.
As usual, $\liminf X_n = -\limsup (-X_n)$.
We also say $\limsup_{n\to\infty} X_n <\infty$ if $\lim_{r\to\infty}\limsup_{n\to \infty} \pp( X_n \ge r) =0$.

In the context of the KPZ equation, where we will always take $a_t=2^{-1/3}t^{1/3}$, the rate function characterizes the associated domain of attraction of the asymptotic fluctuations of $h(t,x)$, which we define next. This fact is a consequence of our main result, see Corollary \ref{cor:domain} below.

\begin{defn}
 Let $h$ and $\tilde h$ be the solutions of the KPZ equation \eqref{eq:KPZ} with initial conditions for \eqref{eq:SHE} given by $\mu$ and $\tilde\mu$. 
 If, for each $x$,
  \[\lim_{n\to \infty}2^{1/3}{t_n}^{-1/3}\!\left(-h(t_n,2^{1/3}t_n^{2/3}x)+\tfrac1{24}t_n\right)
  ~{\buildrel {\rm (d)} \over =} ~ 
  \lim_{n\to \infty}2^{1/3}t_n^{-1/3}\!\left(-\tilde{h}(t_n,2^{1/3}t_n^{2/3}x)+\tfrac1{24}t_n\right)\]
  for every sequence $t_n\nearrow \infty$, then we will say that \emph{$\mu$ and $\tilde\mu$ belong to the same domain of attraction} $\mathscr{D}\subseteq\mathscr{M}$ for the asymptotic one-point fluctuations of the solution of \eqref{eq:KPZ}. 
\end{defn}

The definition should be thought of as saying that two initial conditions belong to the same domain of attraction if their limiting fluctuations coincide; we choose to relax the definition in order to make it suitable to our setting in which only tightness for the fluctuations is known.

Going back to a general setting, let $\big\{\mathsf{P}_t\}_{t>0}$ be a family of probability measures on $\mathcal{C}(\rr)$, which we endow with the topology of uniform convergence on compact sets, and let $X^t$ be a process with paths in $\mathcal{C}(\rr)$ and law $\mathsf{P}_t$.
Similarly, let $\mathsf{P}$ be another probability measure in $\mathcal{C}(\rr)$ and $X$ be a process with this law.

\begin{lem}[Random Laplace asymptotics]\label{lem:lap}
Let $\big\{\nu_t\big\}_{t>0}$ be a family of random measures in $\mathscr{M}$ with laws
$\{\mathsf{Q}_t\}_{t>0}$.
Suppose that $\big\{\nu_t\big\}_{t>0}$ satisfies a large deviation principle with speed $a_t$ and rate function a (possibly random) lower semi-continuous function $I\!: \rr \longrightarrow (-\infty,\infty]$.
Assume that for any $\ep>0$, there are constants $c>0$, $\delta_1\in(0,1)$ and $M_0>0$ such that
\begin{equation}
\label{qbd}
{\mathsf{Q}}\!\left(I(y)\ge -c-(1-\delta_1)M^2~\forall\,|y|\geq M\right)>1-\ep
\end{equation}
for all $M\geq M_0$, where ${\mathsf{Q}}$ is the law of $I$.
Suppose moreover that $\mathsf{P}_t$ (the law of $X^t$) converges in distribution in $\mathcal{C}(\rr)$ to $\mathsf{P}$ (the law of $X$), and that for some $0<\delta_2<\delta_1$,
\begin{equation}\label{abd}
\limsup_{t\to \infty}\,\sup_{y\in\rr} \{X^t(y)-\delta_2y^2\}<\infty
\end{equation}
in distribution.
Then for all $r$, letting $\mathsf{P}_t\!\otimes\!\mathsf{Q}_t$ and $\mathsf{P}\!\otimes\!\mathsf{Q}$ denote the product
measures,
\begin{multline}
\lim_{t\to\infty} \mathsf{P}_t\!\otimes\!\mathsf{Q}_t\!\left(a_t^{-1} \log \int d\nu_t(y)\,e^{a_t[-(x-y)^2+X^t(x-y)] }\leq r\right)\\
= \mathsf{P}\!\otimes\!\mathsf{Q}\!\left(\sup_{y\in\rr} \left\{ X(x-y) - (x-y)^2 - I(y)\right\}\leq r\right).
\end{multline}
\end{lem}

The assumptions \eqref{qbd} and \eqref{abd} can certainly be weakened, but they will be enough for our purposes.

Now we apply this lemma to the problem we are interested in.
First we need to introduce one last family of measures: given any $0\leq\delta<1$ and measures $\mu_t$ as above, we define measures $\mu^\delta_t$ on $\rr$ by
\begin{equation}\label{eq:deltaResc}
d\mu^\delta_t (x)=  e^{-2^{-1/3}\delta t^{1/3}x^2 }\ts d\mu_t(x).
\end{equation}
The $\delta$ is introduced to provide some tightness on initial data such as in Example \ref{bmex}.

Given the initial data $\mu$ for the KPZ equation, we let 
$\mathscr{Q}_t$ denote the law of 
$\mu_t$, defined as in \eqref{eq:mut} and, as in Theorem \ref{thm:CH}, we let $\mathscr{P}_t$ denote the law of $\aip^t$ (as a measure on $\mathcal{C}(\rr)$).
Thus, in view of \eqref{eq:At2}, we may interpret the product probability measure $\mathscr{Q}_t\otimes\mathscr{P}_t$ as the law of $h(t,x)$ started with initial data $\mu_t$.
Note then that the right hand side of \eqref{eq:At2} is set up perfectly for the application of our random version of Laplace's asymptotics.

What follows is the precise statement of our result. 
In simpler, and slightly vague, terms, the result says that if $h(t,x)$ is the solution of the KPZ equation with initial data $Z(0,dx)=\mu(dx)$ such that for some $0\leq\delta<1$ the laws of $\mu^\delta_t$ satisfy a large deviation principle with rate function $\mathfrak{h}_0(x)+\delta x^2$, for some $\mathfrak{h}_0$ which does not decrease too rapidly, then the family of random variables 
\[\left\{2^{1/3}t^{-1/3}\!\left(-h(t,2^{1/3}t^{2/3}x)+\tfrac1{24}t\right)\right\}_{t>0}\]
is tight for every given $x$, and all possible subsequential limits are given, in distribution, by a random variable of the form 
\[\sup_{y\in\rr}\!\big\{ \tilaip(x-y)- (x-y)^2 - \mathfrak{h}_0(y)\big\},\]
with $\tilaip$ the corresponding subsequential limit of $\aipt$.
Of course, if we knew that $\tilaip$ is actually the Airy$_2$ process $\aip$, or in other words that the conjecture \eqref{airylim} holds, then the result could be stated as saying that the above family of random variables converges in distribution to $\sup_{y\in\rr}\!\big\{ \aip(x-y)- (x-y)^2 - \mathfrak{h}_0(y)\big\}$.

\begin{thm}\label{thm:lapl-KPZ}
  Let $\mu\in\mathscr{M}$ satisfy $\mu(\rr)>0$ and consider the solution $h(t,x)$ of the KPZ equation with initial data $Z(0,dx)=\mu(dx)$. Here $\mu$ may be random, but we assume that it is independent of the white noise. Suppose that for some $0\leq \delta<1$ the family of measures $\{\mu^\delta_t\}_{t>0}$ (defined through \eqref{eq:mut} and \eqref{eq:deltaResc}) satisfies the large deviation principle \eqref{eq:ldp} with (random) rate function $\mathfrak h_0(x)+ \delta x^2$ and speed $2^{-1/3}t^{1/3}$.
  Let $\mathscr{Q}$ denote the law of $\mathfrak h_0$, and assume that \eqref{qbd} holds with $I(x)=\mathfrak{h}_0(x)$ and with $\mathscr{Q}$ in place of $\mathsf{Q}$.
  Then for every $x\in\rr$, the family of probability distributions
  \begin{equation}
  \left\{r\in\rr\longmapsto\mathscr{Q}_{t}\!\otimes\!\mathscr{P}_{t}\!\left(\frac{-h(t,2^{1/3}t^{2/3}x)+\inv{24}t}{2^{-1/3}t^{1/3}}\leq r\right)\right\}_{t\geq0}\label{eq:familyProb}
  \end{equation}
  is tight, and any subsequential limit is given by
  \begin{equation}\label{eq:limQP}
    r\in\rr\longmapsto\mathscr{Q}\!\otimes\!\widetilde{\mathscr{P}}\!\left(\sup_{y\in\rr}\left\{ \tilaip(x-y)- (x-y)^2 - \mathfrak{h}_0(y)\right\}\leq r\right),
  \end{equation}
  where $\widetilde{\mathscr{P}}$ is some subsequential limit of $\mathscr{P}_{t_n}$ and $\tilaip$ is the process defined by it.
\end{thm}

\begin{ex}\label{ex:inFn}
  As a simple example suppose that, in the setting of the theorem, we take $\mu(dx)=e^{\lambda(x)}dx$ where $\lambda$ is a stochastic process (independent of the white noise) such that, as $\ep\to0$, $\lambda_\ep(x)=\ep^{1/2}\lambda(\ep^{-1}x)$ converges weakly in the topology of uniform convergence on compact sets to some limiting continuous process $\Lambda$ satisfying $|\Lambda(x)|\leq a+b|x|$ almost surely for some $a,b>0$.
  Then the conclusion of the theorem holds with $\mathfrak{h}_0(x)=-\sqrt{2}\Lambda(x)$.
\end{ex}

\begin{rem}\label{rem:sameNoise}
If we construct two solutions of the KPZ equation (with different initial data) on the same space using the same white noise, and we look at the fluctuations along the same sequence $t_n\nearrow \infty$, then the answer \eqref{eq:limQP} for each solution will be given in terms of the same subsequential limit $\tilaip$ of $\aipt$.
This will be useful below when we compare the limits for different initial data.  
\end{rem}

\begin{rem}\label{rem:CLW}
The answers obtained in \cite{corwinLiuWang} and \cite{chhitaFerrariSpohn} for the asymptotic limiting fluctuations in the setting of last passage percolation and the totally asymmetric exclusion process are given in terms of the same variational problems.
In their case the limiting distributions are written directly in terms of the Airy$_2$ process (as opposed to $\tilaip$), since in their setting the analog of Theorem \ref{thm:CH} (see e.g. \cite{johansson}) gives convergence to $\aip$ instead of only tightness.
\end{rem}

Theorem \ref{thm:lapl-KPZ} and Remark \ref{rem:sameNoise} imply directly that the rate function appearing in the large deviation principle satisfied by the initial data for the KPZ equation characterizes its domain of attraction:

\begin{cor}\label{cor:domain}
  Let $\mu$ and $\tilde \mu$ be two choices of initial data for the KPZ equation (given in terms of the SHE) satisfying the hypotheses appearing in Theorem \ref{thm:lapl-KPZ} for some $0\leq\delta<1$. 
  Assume that the rate functions in the large deviation principle satisfied by the families of rescaled measures $\{\mu^\delta_t\}_{t>0}$ and $\{\tilde\mu^\delta_t\}_{t>0}$ are equal in distribution.
  Then $\mu$ and $\tilde\mu$ belong to the same domain of attraction $\mathscr{D}\subseteq\mathscr{M}$ for the asymptotic fluctuations of solutions of the KPZ equation.
\end{cor}

This gives us a criterion to distinguish the three basic classes: curved, flat and stationary.  We have already described how
$F_{\rm GUE}$ is supposed to govern the asymptotic fluctuations for curved initial conditions and $F_{\rm GOE}$ for flat (for the latter case, see \cite{borFerPrahSasam,bfp} for the derivation starting from periodic TASEP).
For stationary initial data (corresponding, in the context of the KPZ equation, to the initial condition $h(0,x)=B(x)$, a two-sided Brownian motion), the fluctuations are given by 
 $F^\alpha_{\rm stat}$; see \cite{baikRainsF0,ferrariSpohnStat}, which obtain this distribution in the setting of stationary TASEP.
$F^\alpha_{\rm stat}$ is usually referred to as the \emph{Baik-Rains distribution}, especially in the case $\alpha=0$ (in \cite{baikRainsF0} and many other papers $F^0_{\rm stat}$ is denoted as $F_0$).

The way the Tracy-Widom GOE and the Baik-Rains distributions show up in our context is through the variational formulas 
\begin{align}
  F_{\rm GOE}(4^{1/3}r)&=\pp\!\left(\,\sup_{x\in\rr}\big\{\aip(x)-x^2\big\}\leq r\right)\label{eq:GOE},\\
  F^\alpha_{\rm stat}(r)&=\pp\!\left(\,\sup_{x\in\rr}\big\{\aip(x)-(x+\alpha)^2+\sqrt{2}B(x)\big\}\leq r\right)\label{eq:stat}.
\end{align}
The identity \eqref{eq:GOE} was discovered by \citet{johansson}; an alternative proof was given in \cite{cqr} which is intimately related to the general formula we introduce below in Section \ref{sec:hitting}.
The identity \eqref{eq:stat} was proved in \cite{chhitaFerrariSpohn}\footnote{A version of this statement, which does not quite imply \eqref{eq:stat}, was proved earlier in \cite{corwinLiuWang} in the setting of geometric last passage percolation; see the discussion following Corollary 2.8 in that paper.\label{foot:clw}}.
It is known \cite{borCorFerrVeto} to be the limiting distribution of   $t^{-1/3}( -h(t,0)+\frac{t}{24})$ corresponding to the  initial condition $h(0,x)=B(x)$, which is stationary in the sense that $\partial_xh(t,x)$ becomes a stationary in time distribution-valued process with this initial data.

Focusing back on Theorem \ref{thm:lapl-KPZ}, if $\mathfrak{h}_0(0)=0$ and $\mathfrak{h}_0(x)=\infty$ for $x\neq 0$, which corresponds to the narrow wedge initial condition $Z(0,\cdot)=\delta_0$ for the SHE, the right hand side is given simply by the one
dimensional marginal of $\tilaip$, which by \eqref{curvedass} is know to be given by the Tracy-Widom GUE distribution.
On the other hand, in the flat case $\mathfrak{h}_0(y)\equiv 0$ the right hand side would be known to be given by the Tracy-Widom GOE distribution if we knew that $\tilaip$ is indeed the Airy$_2$ process (thanks to \eqref{eq:GOE}).
However, we do know that  $\tilaip(x)$ is locally Brownian \cite{corwinHammondKPZ} and from this it is not hard to see that we have the strict stochastic ordering
 \begin{equation}\label{eq:orderGOEGUE}
\pp\!\left(\sup_{y\in\rr}\left\{ \tilaip(x-y)- (x-y)^2 \right\}\leq r\right)
 > \pp\!\left(\tilaip(0)\leq r\right)=F_{\rm GUE}(r).
\end{equation}

\begin{cor}\label{thm:main}
Recall the setting of Theorem \ref{thm:lapl-KPZ}: we consider the solution $h(t,x)$ of the KPZ equation with initial data $Z(0,dx)=\mu(dx)$, where the (possibly random, but independent of the white noise) measure $\mu\in\mathscr{M}$ satisfies $\mu(\rr)>0$ and is such that for some $\delta\in[0,1)$ the associated family of measures $\{\mu^\delta_t\}_{t>0}$ (defined through \eqref{eq:mut} and \eqref{eq:deltaResc}) satisfies the large deviation principle \eqref{eq:ldp} with (random) rate function $\mathfrak h_0(x)+ \delta x^2$ and speed $2^{-1/3}t^{1/3}$.
Let
\[F^\mu_{t,x}=\pp_\mu\!\left(\frac{-h(t,2^{1/3}t^{2/3}x)+\inv{24}t}{2^{-1/3}t^{1/3}}\leq r\right)\]
(where the subscript in $\pp_\mu$ indicates the initial condition specified above).
Then we have:
\begin{enumerate}[label=\normalfont\arabic*.,itemsep=3pt]
  \item \emph{(Curved)} If $\mathfrak h_0(0)=0$ and $\mathfrak{h}_0(x)=\infty$ for $x\neq 0$, then for all $x\in\rr$,
  \begin{equation} \label{eq:limCurved}
      \lim_{t\to\infty}F^\mu_{t,x}(r)=F_{\rm GUE}(r+x^2).
  \end{equation}
    In particular, the conclusion holds if $\mu$ is  a non-trivial positive measure satisfying 
    \begin{equation}\label{eq:muCurved}
      \mu(\{x\!:|x|>r\})\leq c\ts e^{-\kappa r^{1/2+\delta}}
    \end{equation}
  almost surely, for some $\delta>0$, $\kappa>0$, $c>0$, and all $r>r_0$.     
  \item \emph{(Flat)}  If $\mathfrak h_0\equiv 0$, then the $F^\mu_{t,x}$ are tight, and any subsequential limit is stochastically strictly larger than $F_{\rm GUE}$.
  In particular, the conclusion holds if for some $\ep>0$ the measure $\mu$ satisfies 
      \begin{equation}
          \mu((-r,r))\le c_2\ts e^{\kappa_2 r^{1/2-\delta_2}}\qand     \mu((m-\ep,m+\ep))\ge c_1\tts\ts e^{-\kappa_1 |m|^{1/2-\delta_1}}\label{eq:muFlat}  
      \end{equation}
  for some $\delta_1,\delta_2\in(0,1/2)$, $\kappa_1,\kappa_2>0$, $0<c_1$, $c_2<\infty$, all $r>0$, and all $m\in\rr$ with $|m|$ sufficiently large.
  \item \emph{(Stationary)} If $\mathfrak h_0(x)=\sqrt{2}B(x)$ with $B(x)$ a two-sided Brownian motion, then
  \begin{equation} \label{eq:limStat}
        \lim_{t\to\infty}F^\mu_{t,x}(r)=F^x_{\rm stat}(r),
  \end{equation}
  where $F^x_{\rm stat}$ is the Baik-Rains distribution \eqref{eq:stat}.
  In particular, the conclusion holds if $\mu(dx)=e^{\Theta(x)}dx$ and $\Theta(x)$ is any stochastic process which is independent of the white noise and is such that $\Theta_\ep(x)=\ep^{1/2}\Theta(\ep^{-1}x)$ converges weakly, as $\ep\to0$, in the topology of uniform convergence on compact sets, to a standard Brownian motion $B(x)$.
\end{enumerate}
\end{cor}

We know of course, as we mentioned already, that the limiting distribution in the flat case is $F_{\rm GOE}$, but we cannot identify it rigorously. 
The statement about the stochastic ordering in this case is just \eqref{eq:orderGOEGUE}.

The asymptotics for the stationary case (which can be easily seen to correspond to starting the SHE with $Z(0,x)=e^{-B(x)}$), on the other hand, follow from a combination of Theorem \ref{thm:lapl-KPZ}, Corollary \ref{cor:domain} and the fact, proved in \cite{borCorFerrVeto}, that the one-point asymptotic fluctuations of stationary KPZ equation coincide with those encountered for the stationary TASEP.

We consider next several examples. The first two discuss the distinction between curved and flat in more detail.

\begin{ex}[\bf Curved vs. flat] \label{1}
Let us spell out what the criterion in the first two points of the above corollary means in the case of initial data for the stochastic heat equation taking the natural form:
\begin{equation}
Z(0,x) = e^{ -\kappa |x|^{\alpha} + V(x) }
\end{equation}
where $\alpha >0$ and $V(x)$ is a function with 
$|x|^{-\alpha}|V(x)|\longrightarrow0$ as $x\to\infty$.  \emph{We find that the fluctuations are given
  by $F_{\rm GUE}$ if $\alpha>1/2$ and $\kappa>0$ and by something strictly stochastically larger (conjecturally $F_{\rm GOE}$) if $\alpha<1/2$ and $\kappa\in \rr$.}

\noindent From a more physical point of view, if an interface (with the sign of the non-linearity chosen as in
\eqref{eq:KPZ}) is set up so that asymptotically it has the shape $-\kappa
|x|^{1/2+\delta}$ with $\delta,\kappa>0$, then it is \emph{curved} in the sense that the
fluctuations will have GUE Tracy-Widom law; if it is asymptotically $\kappa
|x|^{1/2-\delta}$ with $\kappa\in\rr$, $\delta>0$, it is \emph{flat} in the sense that the fluctuations  should have GOE
Tracy-Widom law.  This result is not so intuitive and it would be nice to see it verified
experimentally, along the lines of \cite{takeuchiSano1}.
\end{ex}

\begin{ex}[\bf Boundary between curved and flat]\label{3}  
In the critical case where the initial condition for the KPZ equation is asymptotically $-\kappa|x|^{1/2}$, $\kappa>0$, (for instance, if $Z(0,x)=e^{ -\kappa |x|^{1/2} + V(x) }$ with $|x|^{-1/2}|V(x)|\longrightarrow0$ as $x\to\infty$) we identify a new family of crossover distributions, given as follows:
\begin{equation}
F_{\rm sqrt}^{\alpha,b_1, b_2}(r) = \pp\!\left( \sup_{x\in\rr}\!\left\{\aip(x)-x^2 -b_1|x-\alpha|^{1/2}\uno{x<\alpha}-b_2|x-\alpha|^{1/2}\uno{x\geq \alpha})\right\} \le r\right).\label{eq:sqrt-prob2}
\end{equation}
To be more precise, if we start the KPZ equation with the specified initial conditions, then all that Theorem \ref{thm:main} tells us is that the family of probability distributions in \eqref{eq:familyProb} is tight, and all limit points are of the above form with $\aip$ replaced by some subsequential limit $\tilaip$ of the rescaled processes $\aip^t$.
However, as it is widely believed that $\aip^t$ actually converges to $\aip$, it is clear (if still conjectural) that the correct limiting fluctuations are those given in \eqref{eq:sqrt-prob2}, and it is for this reason that we introduce the distributions $F_{\rm sqrt}^{\alpha,b_1, b_2}$ in this way.

\noindent When restricted to $b_1,b_2\in[0,\infty)$ these distributions cross over between $F_{\rm GUE}(r)$ and $F_{\rm GOE}(4^{1/3}r)$:
\begin{equation}\label{eq:sqrtInterp}
	\begin{gathered}
		F_{\rm sqrt}^{\alpha,b_1, b_2}(r) \searrow F_{\rm GOE}(4^{1/3}r) \quad{\rm as}\quad b_1,b_2\searrow 0,\\
		F_{\rm sqrt}^{\alpha,b_1, b_2}(r)\nearrow F_{\rm GUE}(r)\quad{\rm as}\quad b_1,b_2\nearrow \infty.
	\end{gathered}
\end{equation}
It turns out that this new self-similar class also has explicit one-point marginals given by  Fredholm determinants, which we include below as \eqref{eq:sqroot}. 
\end{ex}

\begin{ex}[\bf Crossover cases] One also has crossover families 
\begin{align}
  F^\alpha_{\rm curved\to flat}(r)&=\pp\!\left(\,\sup_{x\leq0}\big\{\aip(x)-(x+\alpha)^2\big\}  \leq r-\alpha^2\uno{\alpha<0}\right)\label{eq:2to1},\\
F^\alpha_{\rm curved\to BM}(r)&=\pp\!\left(\,\sup_{x\leq0}\big\{\aip(x)-(x+\alpha)^2+\sqrt{2}B(x)\big\}\leq r-\alpha^2\right)\label{eq:2tostat},\\
  F^\alpha_{\rm flat\to BM}(r)&=\pp\!\left(\,\sup_{x\in\rr}\big\{\aip(x)-(x+\alpha)^2+\sqrt{2}B(x)\uno{x\leq0}\big\}\leq r-\alpha^2\uno{\alpha<0}\right)\label{eq:1tostat},
\end{align}
which arise from putting two different classes of initial data on each side of the origin.
For example, $F^\alpha_{\rm curved\to flat}$ describes the asymptotic fluctuations for half-flat (or half-periodic) TASEP \cite{bfs}.
It is expected to characterize the KPZ domain of attraction associated to $\mathfrak{h}_0(x) = 0$ for $x\geq0$ and $\mathfrak{h}_0(x) = \infty$ for $x<0$, corresponding to the initial condition $Z(0,x)=\uno{x\geq0}$ for the SHE, but this remains an open problem (see \cite{leDoussalHF} for a non-rigourous argument and \cite{oqr-half-flat} for partial progress on the rigorous side).
The identity \eqref{eq:2to1} was proved in \cite{qr-airy1to2}; the other two can be obtained by a simple extension of the arguments of \cite{chhitaFerrariSpohn} (versions of them also appeared in \cite{corwinLiuWang}, see in footnote \ref{foot:clw} in page \pageref{foot:clw}).
$F^\alpha_{\rm curved\to flat}(r)$ interpolates between $F_{\rm GUE}(r)$ as $\alpha\to- \infty$ and $F_{\rm GOE}(4^{1/3}r)$ as $\alpha\to\infty$.
\end{ex}

\begin{ex}[\bf Beyond the large deviation scale]\label{ex:beyLD} Theorem \ref{thm:lapl-KPZ} is certainly not the last word on fluctuation classes, as one can still have non-trivial phenomena when $\mathfrak{h}_0(x)=-\infty$.  We are basically thinking of $h_0(x) =-f(x)$ where $f$ is some function satisfying $t^{-1/3}f(t^{2/3} x) \to \infty$ in some region.   Already we get something very interesting if $f(x) =x^{\beta}\uno{x>0}$, $\beta\in[1/2,2)$.  Note the obvious requirement $\beta<2$ since we are taking initial data
$Z_0(x) = e^{x^{\beta}\uno{x>0}}$ for the stochastic heat equation.
The right hand side of 
\eqref{eq:linearSHE} at $x=0$ gives in this case
\begin{equation}\label{eq:linearSHE2}
-\log \int_{\rr}e^{-\frac{y^2}{2t} -\frac1{24}t+2^{-1/3}t^{1/3}\mathcal{A}^t(2^{-1/3}t^{-2/3}y) + y^{\beta}\uno{y>0}}\,dy.
\end{equation}
$y^{\beta}-\frac{y^2}{2t}$ is maximized at $y^* = (\beta t)^{ \frac1{2-\beta}}$, around which it looks like 
$c_\beta t^{\frac{\beta}{2-\beta}} -\tfrac1{2t} (2-\beta)(y-y^*)^2$ where $c_\beta=\beta^{\frac{\beta}{2-\beta}} -\tfrac12\beta^{\frac{2}{2-\beta}}$.   
Using the stationarity of $\aip^t$, \eqref{eq:linearSHE2} becomes, up to quadratic order 
\begin{equation}
\tfrac1{24}t-c_\beta t^{\frac{\beta}{2-\beta}}-\log \int_{\rr}e^{-\frac1{2t} (2-\beta)(y-y^*)^2 +2^{-1/3}t^{1/3}\mathcal{A}^t(2^{-1/3}t^{-2/3}(y-y^*)) }\,dy.
\end{equation}
Note that we shifted $\aip^t$, so the approximation is only in distribution.  Note also that if $\beta<1/2$, the shift is within the range of the old optimization problem, so it does nothing. Now we may apply the earlier argument and find that, along subsequences $t_k\nearrow\infty$, the distribution is approximately that of
\begin{equation}
\tfrac1{24}t_k-c_\beta t_k^{\frac{\beta}{2-\beta}}+2^{-1/3} t_k^{1/3} \sup_{x\in\rr}\left\{ \tilaip(x) - (2-\beta)x^2 \right\}.
\end{equation}

\noindent Let us pretend for simplicity in the coming discussion that $\tilaip$ is just $\aip$.
When $\beta=1$ we are in the flat case and the limiting fluctuations are $F_{\rm GOE}$.
Otherwise we identify a new family of distributions depending on the parameter $\beta\in (1/2, 2]$:
\begin{equation}
 F^\beta_{\rm parbl}(r) = \pp\!\left(\sup_x\!\left\{ \aip(x) - (2-\beta)x^2 \right\}\le r\right)
\end{equation} 
(note that the comment following \eqref{eq:sqrt-prob2} also applies here).
Again we have a Fredholm determinant formula, which is presented in Example \ref{ex:parabola}.

\noindent If we start instead with a symmetric version $f(x) =|x|^{\beta}$ ($\beta>1/2$), we end up with two equal maxima at $\pm y*$ and the final formula becomes a maximum over two independent Airy processes.  Note that this does not work in the case $\beta=1/2$ because the maxima are not well separated in that case.  Hence we obtain the same height shift
$\frac1{24}t-c_\beta t^{\frac{\beta}{2-\beta}}$ on top of which we have fluctuations of size $2^{-1/3} t^{1/3}$ with
distribution $(F^\beta_{\rm parbl})^2$.  The case $\beta=1$ corresponds to shock initial data $e^{|x|}$ and the result is the maximum of two independent GOE's.
\end{ex}

\begin{rem}[\bf Ordering] \label{thm:orderings}  From the variational formulas, the following inequalities hold.  For every $\alpha,r\in\rr$ and $b_1>\tilde b_1\geq0$ and $b_2>\tilde b_2\geq0$,
  \begin{gather}    
      F_{\rm GOE}(4^{1/3}r)<F^\alpha_{\rm curved\to flat}(r)<F_{\rm GUE}(r),\\
      \shortintertext{and}
      F_{\rm GOE}(4^{1/3}r)<F^{\alpha,\tilde b_1,\tilde b_2}_{{\rm sqrt}}(r)<F^{\alpha,b_1,b_2}_{{\rm sqrt}}(r)<F_{\rm GUE}(r)
  \end{gather}
  (with similar inequalities for the parabolic case).
  Moreover, if $\tilaip$ is any fixed subsequential limit of $\aipt$ and we replace $\aip$ by $\tilaip$ in the definitions \eqref{eq:GOE}--\ts\tts\eqref{eq:sqrt-prob2}, then the same inequalities hold.  The inequalities are strict because $\tilaip$ is locally Brownian.
  Note the slightly confusing fact that, in general, if we have the stochastic ordering $X\preccurlyeq Y$ as random variables, then their distribution functions satisfy the reverse ordering: 
  $F_X(r) = \pp( X\le r) \ge \pp(Y\le r) = F_Y(r)$. 
  Therefore the above inequalities mean that one has couplings in which the associated random variables are stochastically ordered with respect to one another in the opposite direction.
  
  \noindent In general, if $\mu$ and $\tilde\mu$ are initial data for $Z$, and the family of measures $\mu_t$ and $\tilde\mu_t$ satisfy large deviation principles with rate functions $\mathfrak{h}_0(x)<\mathfrak{\tilde h}_0(x)$ almost surely,
  then $\mu+\tilde\mu$ lies in the same domain of attraction $\mathscr{D}$ as  $\mu$.
 \end{rem}

 \begin{ex}[\bf Compact support]\label{2} 
  Initial data for the KPZ equation \eqref{eq:KPZ} of the form $h_0(x) =  f(x)$ where $\int\!|f|<\infty$ and $f$ has compact support is in the \emph{flat} domain of attraction. 
 \end{ex} 
 
A very natural conjecture is that $F^0_{\rm stat}(r)<F_{\rm GOE}(4^{1/3} r)$, since it would seem that the Brownian 
initial data would swamp the flat initial data.  We were very surprised to find out that this is false. 
This can be easily seen from the tails of these distributions \cite{baikBuckDiF}:

\begin{center}
\begin{tabular}{|c|c|c|}
\hline
$x\searrow -\infty$ & & $x\nearrow \infty$\\
\hline\hline
\mbox{}&&\\[-10pt]
$e^{-\frac1{12}|x|^{3}} $& $F_{\rm GUE} (x) $& $1- e^{-\frac43 x^{3/2}}$ \\[2pt] 
\hline
\mbox{}&&\\[-10pt]
$e^{-\frac16|x|^{3}} $ & $F_{\rm GOE} (4^{1/3} x)$ & $1- e^{-\frac43 x^{3/2}}$ \\[2pt]  
\hline
\mbox{}&&\\[-10pt]
$e^{-\frac1{12}|x|^{3}}$ & $F_{\rm stat}^0 (x)$ & $1- e^{-\frac23 x^{3/2}} $\\[2pt]  
\hline
\end{tabular}
\end{center}

\noindent Along the way, we were fooled into thinking the conjecture is true because we discovered the inequality
\begin{equation}\label{eq:GOEBR}
  F^0_{\rm stat}(r)<F_{\rm GOE}(r).
\end{equation}
However, the meaning of the inequality is now very unclear since it is the distribution function $F_{\rm GOE}(4^{1/3}r)$ which arises in random growth from flat interfaces with our choice of scaling, and \emph{not} $F_{\rm GOE}(r)$.
On the other hand, it is hard to believe that the inequality is meaningless, and it would be very interesting to have some explanation.
Numerically it appears that   $F^0_{\rm stat}(\kappa \tts r)<F_{\rm GOE}(r)$ for any $1\le \kappa\le 2^{1/3}$. For $\kappa\not\in[ 1,2^{1/3}]$ it is easy to check it is not true from the tails.
The inequality \eqref{eq:GOEBR} is proved in Appendix \ref{sec:BRbeatsTW}, based on the representations of both functions in terms of Painlev\'e II transcedents.

The proof of \eqref{eq:GOEBR} also yields the following interesting inequality:
\begin{equation}\label{eq:extraOrd}
  F_{\rm GOE}(r)^4<F_{\rm stat}^0(r)F_{\rm GUE}(r)
\end{equation}
for all $r$.
Noting that
$F_{\rm GOE}(r)^2=\pp\!\left(\,\sup_{x\leq0}\big\{\aip(x)-x^2+\sqrt{2}B(x)\big\}\leq r\right)$
(which follows from a simple extension of the result of \cite{chhitaFerrariSpohn} and known results for half-stationary last passage percolation \cite{baikRainsF0}; see again \cite{corwinLiuWang} for a related identity), this says that
\[\pp\!\left(\,\sup_{x\leq0}\big\{\aip(x)-x^2+\sqrt{2}B(x)\big\}\leq r\right)^2<F_{\rm stat}^0(r)F_{\rm GUE}(r).\]
Apart from providing its proof, Appendix \ref{sec:BRbeatsTW} also discusses the inequality \eqref{eq:extraOrd} in the context of last passage percolation.

\subsection{Airy$_2$ hitting probabilities}\label{sub:airyHit}

Our formulas for $F_{\rm sqrt}^{\alpha,b_1,b_2}$ and $F_{\rm parbl}^\beta$ will follow from an extension of the main result of \cite{cqr}, which is of independent interest, and is one of the main results of this paper.
We are interested in computing $\pp(\aip(t)-t^2\leq g(t)~\forall\,t\in\rr)$ for a broad class of functions $g$.
In view of Theorem \ref{thm:main} and the (conjectural) fact that $\tilaip$ is just the Airy$_2$ process, these probabilities describe the laws associated to the different domains of attraction for the one-point fluctuations of the solution of the KPZ equation.
In particular, they give the one-point distribution of the KPZ fixed point, the conjectural universal limit for all models in the KPZ universality class (see \cite{cqrFixedPt}).

In order to state the general formula we need to introduce some operators.
Define the \emph{Airy kernel} $\K$ as
\begin{equation}
  \label{eq:KAi}
  \K(x,y)=\int_{-\infty}^0 d\lambda\Ai(x-\lambda)\!\Ai(y-\lambda),
\end{equation}
where $\Ai(\cdot)$ is the Airy function. We regard $\K$ as the kernel of an integral
operator, also denoted by $\K$, acting on $L^2(\rr)$.
Define also the \emph{Airy transform} $A$, which is an operator acting on $L^2(\rr)$ with kernel
\begin{equation}
  \label{eq:B0}
  A(x,\lambda)=\Ai(x-\lambda).
\end{equation}
The Airy transform satisfies $AA^*=I$, so that $f(x)=\int_{-\infty}^\infty d\lambda\Ai(x-\lambda)Af(\lambda)$.
In other words, the shifted Airy functions $\{\Ai(x-\lambda)\}_{\lambda\in\rr}$ (which are not in $L^2(\rr)$) form a generalized orthonormal basis of $L^2(\rr)$ (and thus $\K$ is the projection onto the subspace spanned by $\{\Ai(x-\lambda)\}_{\lambda\leq0}$).
It is instructive to think of $A$ as an analogue of the Fourier transform, mapping functions in real space to (Airy) frequency space.

Let $P_m$ denote the projection onto the interval $(m,\infty)$ and $\bar P_m=I-P_m$ 
\begin{equation}
  \label{eq:Pm}
  P_mf(x)=f(x)\uno{x>m}\qqand\bar P_mf(x)=f(x)\uno{x\leq m}.
\end{equation}
With this notation it is clear that 
\begin{equation}\label{eq:K-AiryTr}
 	\K=A\bar P_0A^*.
\end{equation} 
The Airy kernel plays a prominent role in the Fredholm determinant formulas for the Tracy-Widom GUE and
GOE distributions.
In fact, we have
\begin{equation}\label{eq:FGUE}
	F_{\rm GUE}(r)=\det\!\big(I-P_r\K P_r\big)=\det\!\big(I-\K P_r\K\big)
\end{equation}
and
\begin{equation}\label{eq:FGOE}
	F_{\rm GOE}(4^{1/3}r)=\det\!\big(I-\K\varrho_r\K\big)
\end{equation}
where $\varrho_r$ is the reflection operator
\begin{equation}\label{eq:varrho}
	\varrho_rf(x)=f(2r-x).
\end{equation}
The GUE formula is well-known, with the second equality following from the cyclic property of the determinant (and the fact that $\K$ and $P_r$ are projections).
The GOE formula is essentially the one derived by \cite{sasamoto} (and proved in \cite{ferrariSpohn}), see \cite[(1.8)]{cqr}.
The Fredholm determinants here, and in everything that follows, are being computed on the Hilbert space $L^2(\rr)$.
For the definition of the Fredholm determinant and some background (including the definition of the Hilbert-Schmidt and trace class norms to be used below and in Section \ref{sec:hitting}) we refer to \cite[Sec. 2]{quastelRem-review}.

Note how both $F_{\rm GUE}$ and $F_{\rm GOE}$ can be computed as Fredholm determinants of an operator of the form $I-\K R \K$ for some explicit operator $R$.
As we will see next, the same is true for more general probabilities of the form $\pp\big(\aip(t)\leq g(t)+t^2~\forall\,t\in\rr)$.

For given $g\!:[0,\infty)\longrightarrow\rr\cup\{\infty\}$ consider the hitting time
\begin{equation}\label{eq:defTaug}
  \tau_{g}=\inf\{t\geq 0\!:B(t)\geq g(t)\},
\end{equation}
where $B$ is a Brownian motion with diffusion coefficient 2 (this hitting time could, of course, be infinite).
Also, if $g\!:\rr\longrightarrow\rr\cup\{\infty\}$, we define $g^+,g^-\!:[0,\infty)\longrightarrow\rr\cup\{\infty\}$ through
\begin{equation}\label{eq:gplusminus}
	g^+(t)=g(t)\qand g^-(t)=g(-t)\qquad\text{for all $t\geq0$}.
\end{equation}
We will assume that there is a finite or infinite collection of intervals of the form $[a,b]$ with $a\leq b$ such that $g\in H^1([a,b])$ (i.e. $g$ is differentiable in $[a,b]$ and both it and its derivative are in $L^2([a,b])$) and $g(x)=\infty$ for $x$ outside the collection of intervals (note that we could, for instance, take $g$ to be infinite except at some finite collection of points).
We will denote this class of functions as $H^1_{\rm ext}(\rr)$.
  
Using these notations we introduce a new operator as follows:

\begin{defn}\label{defn:brScatt}
  For $g\in H^1_{\rm ext}(\rr)$ such that $g(t)\geq c-\kappa t^2$ for some $\kappa\in[0,3/4)$, we define the kernel
  \begin{equation}\label{eq:defPsig}
  \begin{split}
      \Psi^g(\lambda,y)&=A^*(\lambda,y)-\int_{0}^{\infty}\,\pp_y(\tau_{g}\in dt)e^{-t\Delta}A^*(\lambda,g(t))\\
      &=\Ai(y-\lambda)-\int_{0}^{\infty}\,\pp_y(\tau_{g}\in dt)e^{-2t^3/3-(g(t)-\lambda)t}\Ai(g(t)-\lambda+t^2)
  \end{split}
  \end{equation}
  and then the \emph{Brownian scattering operator} $\scat^{g}$ as
  \begin{equation}\label{eq:defScatt}
    \scat^g=\Psi^{g^-}\bar P_{g(0)}\ts(\Psi^{g^+})^*,
  \end{equation}
  or, in other words,  
  \begin{equation}\label{eq:defScatgAlt}
  	\scat^g(\lambda_1,\lambda_2)=\int_{(-\infty,g(0))}dy\,\Psi^{g^-}(\lambda_1,y)\Psi^{g^+}(\lambda_2,y).
  \end{equation}
  More generally, given any $\alpha\in\rr$, if we let $g_\alpha(t)=g(t+\alpha)$ and define
  \begin{equation}
  \label{eq:defPsigAlpha}
    \begin{split}
      &\Psi^g_\alpha(\lambda,y)=e^{-\alpha \Delta}A^*(\lambda,y)-\int_{0}^{\infty}\pp_y(\tau_{g_\alpha}\in
        dt)A^*e^{-(\alpha+t)\Delta}(\lambda,g_\alpha(t))\\
      &\qquad=e^{-2\alpha^3/3-\alpha(y-\lambda)}\Ai(y-\lambda+\alpha^2)\\
      &\qquad\qquad-\int_{0}^{\infty}\pp_y(\tau_{g_\alpha}\in dt)e^{-2(\alpha+t)^3/3-(\alpha+t)(g_\alpha(t)-\lambda)}\Ai(g_\alpha(t)-\lambda+(\alpha+t)^2),
    \end{split}
    \end{equation}
  then $\scat^g$ can also be written explicitly as
  \begin{equation}\label{eq:defScatgAlpha}
  \scat^g(\lambda_1,\lambda_2)=\Psi^{g^-_\alpha}_{-\alpha}\ts\bar P_{g(\alpha)}\ts(\Psi^{g^+_\alpha}_\alpha)^*(\lambda_1,\lambda_2)=\int_{(-\infty,g(\alpha))}dy\,\Psi^{g^-_\alpha}(\lambda_1,y)\Psi^{g^+_\alpha}(\lambda_2,y).
  \end{equation}
  If $g$ is not continuous at $\alpha$, then $g(\alpha)$ in the above formulas should be interpreted as $\min\{g(\alpha-),g(\alpha+)\}$.
\end{defn}

The fact that the right hand side of \eqref{eq:defScatgAlpha} does not depend on $\alpha$ will be proved in Section \ref{sec:hitting} (see Proposition \ref{prop:alphaIndep}).

Although the inverse heat kernels appearing in the integrand in the definition of $\Psi^g_\alpha$ may in principle appear to be nonsensical, they are in fact well-defined in this setting, because $A^*e^{s\Delta}$ is well-defined (and in $L^2(\rr)$) for all $s\in\rr$.
See Proposition \ref{prop:mLaplacian} below, which also justifies the second equality in \eqref{eq:defPsig} and in \eqref{eq:defPsigAlpha}.

We remark that, under our assumptions on $g$, $\K-A\bar P_0\scat^{g}\bar P_0A^*$ is a trace class operator mapping $L^2(\rr)$ to itself (this is proved as Proposition \ref{prop:OmegaInfty-TrCl}).
This implies that the Fredholm determinant in the next theorem is well-defined.

\begin{rem}\label{rem:barPg0}
In \eqref{eq:defScatgAlt}/\eqref{eq:defScatgAlpha} we have written the domain of integration in $y$ as $(-\infty,g(\alpha))$ to stress that $g(\alpha)$ is not included in the integration (even though $\bar P_{g(\alpha)}$ is the projection onto $(-\infty,g(\alpha)]$, which could lead to confusion).
This is so because $\Psi^{g^\pm_\alpha}(\lambda,g(\alpha))=0$.
In fact, if $y=g(\alpha)$ then $\pp_y(\tau_{g^\pm_\alpha}\in dt)$ is, by definition of $\tau_g$, just a point mass at $t=0$, so that
\[\int_{0}^{\infty}\,\pp_{g(\alpha)}(\tau_{g^\pm_\alpha}\in dt)e^{-(\alpha+t)\Delta}A^*(\lambda,g_\alpha(t))=e^{-\alpha \Delta}A^*(\lambda,g(\alpha)).\]
Note that $e^{-(\alpha+t)\Delta}A^*(\lambda,g(t))$ vanishes rapidly as $t\to\infty$, so $\tau_{g^\pm_\alpha}$ possibly being $\infty$ does not affect the integral.
\end{rem}

At this point we are ready to state our general continuum statistics formula for the Airy$_2$ process.

\begin{thm}\label{thm:hittProb}
  Given $g\in H^1_{\rm ext}(\rr)$ such that $g(t)\geq c-\kappa t^2$ for some $\kappa\in[0,3/4)$, we have
  \begin{equation}\label{eq:hittProb}
    \pp\!\left(\aip(t)\leq g(t)+t^2~\forall\,t\in\rr\right)=\det\!\Big(I-\K+A\bar P_0\scat^{g}\bar P_0A^*\Big)
  \end{equation}
  with $\scat^g$ given as in Definition \ref{defn:brScatt}.
\end{thm}

A series of remarks are in order.

\makeatletter
\newcommand{\mylabel}[2]{\def\@currentlabel{#2}\label{#1}}
\makeatother

\begin{rem}\label{rem:long}
\mbox{}\\[-16pt]
\begin{enumerate}[label=\arabic{*}.,itemsep=3pt]

  \item The requirement that $\kappa<3/4$ in the above theorem is purely technical, and we expect the result to hold for all $\kappa<1$.
	Nevertheless, since $\kappa<3/4$ is enough for the main examples we have in mind in this paper, we will not attempt to improve this.

  \item The importance of the formula with respect to the earlier one in \cite{cqr} is that the earlier formula was restricted to a finite interval.
  In general, the distributions arising in the long time limit have the form  $\sup_t\left\{ \aip(t) - t^2 -g(t) \right\}$, so the earlier formula required initial data $g(\cdot)$ which was infinite outside a finite interval.
  The distributions were obtained by taking a limit of large intervals, but these limits turn out to be rather singular and have to be performed on a case by case basis.

  \item The analogy between \eqref{eq:hittProb} and \eqref{eq:FGUE}/\eqref{eq:FGOE} can be made more explicit by writing 
  \begin{gather}
    \textstyle\Psi^g=A^*\check{\Psi}^g\qqand\check{\scat}^g=\check{\Psi}^{g^-}\bar P_{g(0)}\ts(\check{\Psi}^{g^+})^*\\
    \text{with}\qquad\check\Psi^g(x,y)=\delta_0(x-y)-\int_{0}^{\infty}\pp_y(\tau_{g}\in dt)e^{-t\Delta}(x,g(t)),
  \end{gather}
  so that in view of \eqref{eq:defScatt} we have
  \begin{equation}\label{eq:hittProbAlt}
    \pp\!\left(\aip(t)\leq
        g(t)+t^2~\forall\,t\in\rr\right)=\det\!\Big(I-\K(I-\check{\scat}^{g})\K\Big)
  \end{equation}
  (this formula could be made rigorous by defining properly the domain and range of the operators $\check{\Psi}^g$, but we opt not to dwell on this technicality since it does not really provide any better understanding of \eqref{eq:hittProb}).
  \vskip2pt
  \noindent As the examples that follow this remark show, one can use \eqref{eq:hittProb}/\eqref{eq:hittProbAlt} to recover the Tracy-Widom GUE and GOE distributions in the curved and flat cases, respectively.

  \item The version of the formula for $\scat^g$ given in \eqref{eq:defScatgAlpha} is useful whenever $g$ is naturally split at some $\alpha\in\rr$, see e.g. Examples \ref{ex:2to1fromGral} and \ref{ex:sqrtBd}.
  In the second of these two examples it is convenient to use a slightly modified version of \eqref{eq:hittProb}, which is obtained by conjugating the operator inside the determinant by $e^{\alpha H}$ with $H$ as defined in \eqref{eq:airyHam} (in \eqref{eq:hittProbAlpha2} below we also shift the kernel by $\alpha^2$).
  It reads as follows: if
  \begin{equation}\label{eq:tildePsi}
  	\begin{aligned}
  		&\wt\Psi^g_\alpha(\lambda,y)=A^*e^{-\alpha\xi}(\lambda,y)-\int_{0}^{\infty}\pp_y(\tau_{g_\alpha}\in
  	  dt)A^*e^{-\alpha\xi}e^{-t\Delta}(\lambda,g(t))\\
  		&\hspace{0.3in}=e^{-\alpha y}\tsm\Ai(y-\lambda)-\int_{0}^{\infty}\pp_y(\tau_{g_\alpha}\in dt)e^{-2t^3/3-2\alpha t^2-\alpha^2t-(\alpha+t)g_\alpha(t)+\lambda t}\\
  		&\hspace{2.8in}\times \Ai(g_\alpha(t)-\lambda+t^2-2\alpha t),
    \end{aligned}
  \end{equation}
  $\wt\scat^g_\alpha=\wt\Psi^{g^-_\alpha}_{-\alpha}\ts\bar P_{g(\alpha)}\ts(\wt\Psi^{g^+_\alpha}_\alpha)^*$ and $\K^\alpha=A\bar P_{-\alpha^2}A^*$, then
  \begin{equation}\label{eq:hittProbAlpha2}
    \pp\!\left(\aip(t)\leq g(t)+t^2~\forall\,t\in\rr\right)=\det\!\Big(I-\K^\alpha+A\bar P_{-\alpha^2}\wt\scat^{g}_\alpha\bar P_{-\alpha^2}A^*\Big).
  \end{equation}

\end{enumerate}
\end{rem}

Theorem \ref{thm:hittProb}  gives  explicit Fredholm determinant formulas for $\pp(\aip(t)\leq g(t)+t^2~~\forall\,t)$ whenever the hitting time density of a Brownian motion from any point $y<g(\alpha)$ to the curve $g$ is explicit.
We consider next several choices of $g$ for which this is possible.
While in the first three examples the theorem is used to recover known results, the last two examples provide new formulas.

\begin{ex}\label{ex:GUEfromGral} {\bf (Curved/GUE)}
  \mbox{~~}Consider the case $g(0)=r$ and $g(t)=\infty$ for all $t\neq0$ (which corresponds to computing the marginal of $\aip(t)$ at time $t=0$).
  Note that the hitting densities appearing in \eqref{eq:defPsig} are zero in this case for all $t\geq0$ and $y<g(0)=r$, and so from \eqref{eq:defScatt}/\eqref{eq:defScatgAlt} we trivially get $\Psi^{g^\pm}=A^*$ and thus $\scat^{g}=A^*\bar P_rA$, so that \eqref{eq:hittProb} becomes 
  \[\pp(\aip(0)\leq r)=\det\!\Big(I-\K+A\bar P_0A^*\bar P_rA\bar P_0A^*\Big)=\det\!\Big(I-\K P_r\K\Big)=F_{\rm GUE}(r)\]
  (where we used \eqref{eq:K-AiryTr}) as desired.
\end{ex}

\begin{ex}\label{ex:GOEfromGral} {\bf (Flat/GOE)}
  \mbox{~~}The flat case, corresponding to $g\equiv r$, is a bit more involved, and relies on showing that in this case $\Psi^{g^+}=\Psi^{g^-}=A^*(I-\varrho_r)$, where $\varrho_r$ is the reflection operator $\varrho_rf(x)=f(2r-x)$.
  This implies that $\scat^g=A^*(I-\varrho_r)\bar P_r(I-\varrho_r)A=A^*(I-\varrho_r)A$.
  Using this and \eqref{eq:K-AiryTr} again, \eqref{eq:hittProb} becomes
  \begin{equation}\label{eq:GOEfromGral}
    \begin{aligned}
    \pp(\aip(t)\leq r+t^2~~\forall t\in\rr )&=\det\!\Big(I-\K+A\bar P_0A^*(I-\varrho_r)A\bar P_0 A^*\Big)\\
    &=\det\!\Big(I-\K\varrho_r\K\Big)=F_{\rm GOE}(4^{1/3}r)
  \end{aligned}
  \end{equation}
  as desired.   
  The proof of the formula for $\Psi^{g^\pm}$ will be provided at the end of Section \ref{sec:hitting} (see \eqref{eq:flatPsi}).
\end{ex}

\begin{ex}\label{ex:2to1fromGral} {\bf (Curved$\to$Flat/GUE$\to$GOE)}
\mbox{~~}Consider now the case $g(t)=r$ for $t\leq\alpha$ and $g(t)=\infty$ for $t>\alpha$.
In this case it is clearly convenient to split $g$ at $\alpha$, as in \eqref{eq:defScatgAlpha}.
We will actually use \eqref{eq:hittProbAlpha2}.
As in Example \ref{ex:GUEfromGral} we trivially have $\wt\Psi^{g^+_\alpha}_\alpha=A^*e^{-\alpha\xi}$ while, by \eqref{eq:flatPsi}, we have $\wt\Psi^{g^-_\alpha}_{-\alpha}=A^*e^{\alpha\xi}-A^*e^{\alpha\xi}\varrho_r$.
Therefore $\wt\scat^g_\alpha=(A^*e^{\alpha\xi}-A^*e^{\alpha\xi}\varrho_r)\bar P_re^{-\alpha\xi}A=A^*\bar P_rA-A^*e^{\alpha\xi}\varrho_r\bar P_re^{\alpha\xi}A$, and thus (using $(\K^\alpha)^2=\K^\alpha$)
\begin{align}
	\pp(\aip(t)\leq r+t^2~~~\forall t\leq\alpha )&=\det\!\Big(I-\K^\alpha P_r\K^\alpha-\K^\alpha e^{\alpha\xi}\varrho_r\bar P_re^{-\alpha\xi}\K^\alpha\Big)\\
	&=\det\!\Big(I-\K P_{r+\alpha^2}\K-\K e^{\alpha\xi}\varrho_{r+\alpha^2}\bar P_{r+\alpha^2}e^{-\alpha\xi}\K\Big),
\end{align}
where in the second equality we conjugated by $e^{-\alpha^2\nabla}$ and used $e^{-\alpha^2\nabla}\varrho_re^{\alpha^2\nabla}=\varrho_{r+\alpha^2}$.
This formula coincides with the formula given in \cite[(2.19)]{qr-airy1to2}, and as a consequence (see Theorem 1 in the same paper) we obtain
\[\pp(\aip(t)\leq r+t^2~~~\forall t\leq\alpha )=F_{{\rm curved}\to{\rm flat}}(r+\min\{0,\alpha\}^2),\]
which was introduced in \eqref{eq:2to1}.
\end{ex}

\begin{ex}\label{ex:sqrtBd} {\bf (Square root barrier)} 
  \mbox{~~}We consider now the case of a square root barrier which, as discussed in Example \ref{3}, correponds to the critical case between the curved and flat domains of attraction.
  Recall that we defined (in \eqref{eq:sqrt-prob2})
  \[F_{\rm sqrt}^{\alpha,b_1, b_2}(r) = \pp\Big( \sup_{x\in\rr}\!\left\{\aip(x)-x^2 -b_1|x-\alpha|^{1/2}\uno{x<\alpha}-b_2|x-\alpha|^{1/2}\uno{x\geq \alpha})\right\} \le r\Big).\]
  We will restrict the discussion to the case $b_1,b_2>0$.
  By symmetry it suffices to compute $\Psi^{g}_\alpha$, where $g(x)=b\sqrt{x}+r$ for $x\geq0$ (and $b>0$).
  The hitting probability $\pp_y(\tau_{g}\in dt)$ can be obtained from \cite[eqn.\ts (27)]{donchev} (see also \cite{novFrishKord}), who proves that if $W(t)$ is a standard Brownian motion and $\hat\tau_{a,b}$ is the first hitting time of $a+b\sqrt{t}$ by $W(t)$ then
  \begin{equation}\label{eq:donchev}
  \pp(\hat\tau_{a,b} \in dt) = \sum_{n=1}^\infty \frac{a^{2\hat\nu_n(b)}}{\partial_\nu \Upsilon(\nu,\sqrt{2}b)\big|_{\nu=\hat\nu_n(b)}}t^{\hat\nu_n(b)-1}\,dt,
  \end{equation}
  where 
  \[\Upsilon(\nu,z)=2^\nu e^{z^2/4}D_{-2\nu}(z)=\frac{\sqrt{\pi}}{\Gamma(\tfrac12+\nu)}\ts {}_1\tsm F_1(\nu,\tfrac12;\tfrac12z^2)-\frac{\sqrt{2\pi}z}{\Gamma(\nu)}\ts {}_1\tsm F_1(\tfrac12+\nu,\tfrac32;\tfrac12z^2)\]
  with $D_{p}(z)$ the parabolic cylinder function defined in \cite[(12.2.5)]{NIST:DLMF} (for the second equality as well as integral representations of $D_p(z)$ see \cite[Sec. 9.24]{gradRyzh}) and $\big(\hat\nu_n(b)\big)_{n\geq1}$ the (negative) real zeros of the function $\Upsilon(\cdot,\sqrt{2}b)$ (in writing \eqref{eq:donchev} from the formula provided in \cite{donchev} we are using $\lim_{c\to0}c^{-\nu}\Upsilon(\nu,a/\sqrt{c})=a^{2\nu}$, as follows from \cite[(12.9.1)]{NIST:DLMF}).
  Applying a simple scaling argument to go back to Brownian motions with diffusivity 2 we get the following formula for the density of $\tau_g$:
  \begin{equation}
  \rho^{\rm sqrt}_{y,r,b}(t):=\frac{\pp_y(\tau_g \in dt)}{dt}
  =\sum_{n=1}^\infty \frac{\left(\tfrac12(y-r)\right)^{2\nu_n(b)}t^{\nu_n(b)-1}}{\partial_\nu \Upsilon(\nu,b/\sqrt{2})\big|_{\nu=\nu_n(b)}}
  \end{equation}
  where $\big(\nu_n(b)\big)_{n\geq1}$ are the (negative) real zeros of the function $\Upsilon(\cdot,b/\sqrt2)$.
  In view of this and \eqref{eq:hittProb} we get
  \begin{equation}\label{eq:sqroot}
    F_{\rm sqrt}^{\alpha,b_1,b_2}(r)=\det\!\Big(I-\K^\alpha+A\bar P_{-\alpha^2}\scat^{{\rm sqrt}}_{\alpha,b_1,b_2}\bar P_{-\alpha^2}A^*\Big),
  \end{equation}
  where $\scat^{{\rm sqrt}}_{\alpha,b_1,b_2}=\Psi^{\rm sqrt}_{-\alpha,b_1}\bar P_r(\Psi^{\rm sqrt}_{\alpha,b_2})^*$ with
  \begin{multline}
  	\Psi^{\rm sqrt}_{\alpha,b}(\lambda,y)=e^{-\alpha y}\Ai(y-\lambda)-\int_0^\infty dt\,\rho^{\rm sqrt}_{y,r,b}(t)e^{-2t^3/3+2\alpha t^2-\alpha^2t+(\alpha-t)b\sqrt{t}-\lambda t}\\
  	\times\Ai(b\sqrt{t}-\lambda+t^2-2\alpha t).
  \end{multline}
  This formula can in principle also be used directly to verify the limits in \eqref{eq:sqrtInterp}.
  In fact, and by definition, $\rho^{\rm sqrt}_{y,r,b}$ converges as $b\to0$ to the density of the hitting time of $g\equiv r$ by a Brownian motion started at $y$, while as $b\to\infty$ it converges to 0 for all $y<r$ and all $t\geq0$.
    Therefore (without checking the details about the convergence) we deduce as in Examples \ref{ex:GUEfromGral} and \ref{ex:GOEfromGral} that \[F_{\rm sqrt}^{\alpha,b_1,b_2}(r)\xrightarrow[b_1,b_2\to0]{}F_{\rm GOE}(4^{1/3}r)\qqand F_{\rm sqrt}^{\alpha,b_1,b_2}(r)\xrightarrow[b_1,b_2\to\infty]{}F_{\rm GUE}(r).\]
\end{ex}

\begin{ex}\label{ex:parabola} {\bf (Parabolic barrier)}
	\mbox{~~}We consider finally the case of a general parabolic barrier, 
	\begin{equation}
		F_{\rm parbl}^{\beta_1,\beta_2}(r)=\pp\Big(\sup_{x\in\rr}\!\left\{\aip(x)-(1+\beta_1)x^2\uno{x<0}-(1+\beta_2)x^2\uno{x>0})\right\}\le r\Big),
	\end{equation}
	introduced in Example \ref{ex:beyLD}.
	Of course, one can also consider more general cases, for instance by splitting the parabolic barrier at a point other than the origin (as in the previous example), but for simplicity we will stick to this version.
	We will assume additionally that $\beta_1,\beta_2>0$; similar formulas can be written if the $\beta_i$'s are negative, but require different formulas for the Brownian hitting times (see e.g. \cite{martin-lof}).

	\noindent To obtain a formula for $F_{\rm parbl}^{\beta_1,\beta_2}(r)$ we need an expression for the hitting probability $\pp_y(\tau_g\in dt)$ with $g(t)=r+\beta t^2$.
	This is a classical problem, which has appeared repeatedly in the literature.
	We use a formula derived in \cite{salminen}, which we take from \cite[(3.2)]{jansonLouchardMartin-lof} and says that if $\hat\tau_{x,b}$ is the hitting time of $b t^2$, $b>0$, by a standard Brownian motion started at $x<0$, then
	\[\pp(\hat\tau_{x,b} \in dt)=2^{1/3}b^{2/3}\sum_{k\geq1}e^{2^{1/3}b^{2/3}a_kt-\frac23b^2t^3}\frac{\Ai(-2^{2/3}b^{1/3}x+a_k)}{\Ai'(a_k)}dt,\]
	where $\{a_k\}_{k\geq 1}$ are the zeros of the Airy function (which are all negative).
	By Brownian scaling, the hitting time $\pp_y(\tau_g\in dt)$ (which, we recall, is defined in terms of a Brownian motion with diffusion coefficient 2) has the same distribution as $2^{-1/3}\beta^{-2/3}\hat\tau_{2^{-1/3}\beta^{1/3}(y-r),\frac12}$, and thus
	\begin{equation}\label{eq:salmon}
	  \rho^{\rm parbl}_{y,r,\beta}(t):=\frac{\pp_y(\tau_g \in dt)}{dt}
  	  =\beta^{2/3}\sum_{k\geq1}e^{\beta^{2/3}a_kt-\frac13\beta^2t^3}\frac{\Ai(\beta^{1/3}(r-y)+a_k)}{\Ai'(a_k)}.
	\end{equation}
	Alternatively, we may use \cite[Lem. 3.2 and Rem. 3.3]{jansonLouchardMartin-lof} to write
	\begin{equation}\label{eq:salmon2}
	  \rho^{\rm parbl}_{y,r,\beta}(t)=\beta^{2/3}\frac1{2\pi\I}\int_{\I\rr} dz\, e^{\beta^{2/3}zt - \frac13\beta^2t^3}\frac{\Ai(z+\beta^{1/3}(r-y))}{ \Ai (z)}.
	\end{equation}
	We get
   \begin{equation}\label{eq:parabola}
     F_{\rm parbl}^{\beta_1,\beta_2}(r)=\det\!\Big(I-\K+A\bar P_{0}\scat_{{\rm parbl}}^{\beta_1,\beta_2}\bar P_{0}A^*\Big),
   \end{equation}
   where $\scat_{{\rm parbl}}^{\beta_1,\beta_2}=\Psi^{\rm par}_{\beta_1}\bar P_r(\Psi^{\rm par}_{\beta_2})^*$ with
   \begin{equation}
   	\Psi^{\rm par}_{\beta}(\lambda,y)=\Ai(y-\lambda)-\int_0^\infty dt\,\rho^{\rm parbl}_{y,r,\beta}(t)e^{-2t^3/3-\beta t^3+\lambda t}\Ai((\beta+1)t^2-\lambda).
   \end{equation}

  \noindent As in Example \ref{ex:sqrtBd}, the formula can be used to verify that $F_{\rm parbl}^{\beta_1,\beta_2}(r)$ interpolates between  $F_{\rm GUE}$ and $F_{\rm GOE}$.
  Moreover, in this case it is not hard to do this explicitly (although we will only sketch the argument, skipping the technical details).
  On the one hand, since $a_k<0$ for all $k\geq1$, it is easy to see from \eqref{eq:salmon} that $\rho^{\rm parbl}_{y,r,\beta}(t)\xrightarrow[\beta\to\infty]{}0$ for all $t\geq0$ and all $y<r$, and thus
  \[F_{\rm parbl}^{\beta_1,\beta_2}(r)\xrightarrow[\beta_1,\beta_2\to\infty]{}F_{\rm GUE}(r)\]
  as in Example \ref{ex:GUEfromGral}.
  To deal with the case $\beta\to0$ we use \eqref{eq:salmon2} and rescale to get, setting $\hat{y}=r-y$,
  \begin{align}
  	\rho^{\rm parbl}_{y,r,\beta}(t)&=\frac1{2\pi\I}\int_{\I\rr}dz\,e^{tz-\frac13\beta^2 t^3} \frac{ \Ai (\beta^{-2/3} (z+\beta\hat y )}{ \Ai (\beta^{-2/3} z)}\\
    &=\frac1{2\pi\I}\int_{\I\rr}dz\,e^{tz - [ \frac23 (\beta^{-2/3} (z+\beta\hat{y} ))^{3/2} - \frac23 (\beta^{-2/3} z)^{3/2}]+o(1)}
    =\frac1{2\pi\I}\int_{\I\rr}dz\,e^{tz - z^{1/2}\hat{y}+o(1)},
  \end{align}
  where we used the asymptotics of the Airy function \eqref{eq:airyBd} and by $o(1)$ we mean a term which goes to 0 as $\beta\to0$.
  We change variables $u\mapsto z^{1/2}$, which turns the contour $\I\rr$ into the contour $\Gamma$ composed of the two rays starting at the origin with angles $\pm\pi/4$.
  We get
  \[\rho^{\rm parbl}_{y,r,\beta}(t)=\frac{1}{\pi\I}\int_{\Gamma}du\,u \ts e^{tu^2 - u \hat{y}+o(1)}=\frac{\hat y}{2\pi\I t}\int_{\Gamma}du\,e^{tu^2 - u \hat{y}+o(1)}\]
  as $\beta\to0$, where in the second equality we used $\partial_u e^{tu^2 - u \hat{y}} = (2tu -\hat{y}) e^{tu^2 - u \hat{y}}$ and the fact that $e^{tu^2 - u \hat{y}}$ vanishes at either end of the contour $\Gamma$.
  Now we change variables $u=(2t)^{-1/2} v$ to get
  \[\rho^{\rm parbl}_{y,r,\beta}(t)\xrightarrow[\beta\to0]{}\tfrac{\hat{y}}{(2t)^{3/2}\pi\I} \int_{\Gamma}dv\,e^{\frac{v^2}2 - v(2t)^{-1/2} \hat{y}}
	=\tfrac{\hat{y}}{(2t)^{3/2}\pi }\int_{-\infty}^\infty\!dw\, e^{-\frac{w^2}2 - \I w(2t)^{-1/2} \hat{y}}
    =\tfrac{\hat{y}}{2\sqrt{\pi}t^{3/2} }e^{ -\frac{\hat{y}^2}{4t}},\]
  where in the first equality we deformed the contour $\Gamma$ back to $\I\rr$ and changed variables $v\mapsto \I w$.
  The result coincides, as expected, with the density of the hitting time to $r$ of a Brownian motion started at $y<r$, and hence as $\beta_1,\beta_2\to0$ we recover the flat case (Example \ref{ex:GOEfromGral}):
  \[F_{\rm parbl}^{\beta_1,\beta_2}(r)\xrightarrow[\beta_1,\beta_2\to0]{}F_{\rm GOE}(4^{1/3}r).\]
\end{ex}

\section{Proofs of the results from Section \ref{sub:criterion}}\label{ref:pfMainThm}

\subsection{Proof of Lemma \ref{lem:lap}}

We begin with the proof of the random Laplace asymptotics.
Without loss of generality we may assume that $a_t=t$.
Assuming that the measures $\nu_t$ are deterministic, standard arguments (see e.g. \cite{varadhanCBMS}) show that for any compact $K$,
\begin{equation}
  \lim_{t\to\infty} t^{-1} \, \log \int_{K} e^{t[-(x-y)^2+X^t(x-y)] }\, d\nu_t(y) 
  \,\stackrel{({\rm d})}{=}\, \sup_{y\in K} \big\{ X(x-y)-(x-y)^2- I(y)\big\}.
\end{equation}
This limit can be extended easily to our setting by conditioning, since in our case the measures $\nu_t$ are independent of $X^t$.
Therefore, since $\log(a\vee b)\leq\log(a+b)\leq\log(2(a\vee b))$ for all $a,b>0$, and choosing $K$ to be of the form $[x-M,x+M]$, it suffices to show that, under our conditions,
 \begin{equation}\label{eq:LDPreduct}
   \limsup_{M\to\infty} \limsup_{t\to\infty} t^{-1} \log \int_{|y-x|>M} e^{t[-(x-y)^2+X^t(x-y)]}\, d\nu_t(y)  = -\infty
\end{equation}
in distribution.
That is, we need to show that given any $\ep>0$ and $\ell>0$,
\begin{equation}\label{eq:needRLA}
  \mathsf{P}_t\!\otimes\!\mathsf{Q}_t\!\left(  t^{-1}\log \int_{|y-x|>M} e^{t[-(x-y)^2+X^t(x-y)]}\, d\nu_t(y)\ge -\ell\right)\leq\ep
\end{equation}
for $M>M_0$, $t>t_0$ and large enough $M_0>0$ and $t_0>0$.

Recalling our assumptions \eqref{qbd} and \eqref{abd}, we may choose $\bar\delta\in(\delta_2,\delta_1)$ and some large enough $M_0$ and $t_0$ so that
\begin{gather}\label{eq:infh0}
  {\sf Q}\Big(\inf_{|y-x|\geq M}I(y)\leq\ell-(1-\bar\delta)M^2\Big)<\ep/2\\
  \shortintertext{and}
  \mathsf{P}_t \bigg(\sup_{|y-x|>M} \big\{X^t(x-y)- \bar\delta(x-y)^2\big\}\le 0\bigg)\geq1-\ep/2.
\end{gather}
for $t\geq t_0$, $M\geq M_0$.
On the other hand, if $X^t(x-y)\leq \bar\delta(x-y)^2$ for $|y-x|>M$ then
\begin{multline}
  t^{-1}\log \int_{|y-x|>M} e^{t[-(x-y)^2+X^t(x-y)]}\, d\nu_t(y)\\
  \le -(1-\bar\delta)\inf_{|y-x|>M} (x-y)^2+ t^{-1}\log \nu_t((x-M,x+M)^\text{c}).
\end{multline}
Hence, since $\nu_t$ and $X^t$ are independent, we get for any $t\ge t_0$ that
 \begin{multline}
  \mathsf{P}_t\!\otimes\!\mathsf{Q}_t\!\left(  t^{-1}\log \int_{|y-x|>M} e^{t[-(x-y)^2+X^t(x-y)]}\, d\nu_t(y)\ge -\ell\right ) \\
   \le\mathsf{Q}_t\Big(t^{-1}\log \nu_t((x-M,x+M)^\text{c})  \ge -\ell+(1-\bar\delta)M^2\Big) + \ep/2 .
\end{multline}
Taking $t\to\infty$ and using the assumption that $\mathsf{Q}_t$ satisfies a large deviation principle with respect to $I$, the right hand side is bounded by
\begin{equation}
 \overline{\mathsf{Q}}\!\left(-\inf_{|y-x|\geq M}I(y)\ge -\ell+(1-\bar\delta)M^2\right) +\ep/2\leq\ep,
\end{equation}
where in the last inequality we have used \eqref{eq:infh0}.
This finishes our proof of \eqref{eq:needRLA} and of the lemma.

\subsection{Proof of Theorem \ref{thm:lapl-KPZ}}

By Theorem \ref{thm:CH}, we know that the family of processes $(\aipt)_{t>0}$ is tight.
Let $\tilaip$ be a subsequential limit, say $\mathcal{A}^{t_n}\longrightarrow\tilaip$ when $n\to \infty$ (with $t_n\nearrow \infty$) in distribution.
We claim that it is enough to show that for all $\kappa>0$,
\begin{equation}\label{eq:pvBd}
\lim_{r\to \infty}\lim_{n\to \infty}\tsm\pp\!\left(\sup_{x\in\rr}\left\{\mathcal{A}^{t_n}(x)-\kappa x^2\right\}\geq r\right)=0.
\end{equation}
In fact, using \eqref{eq:At2} and noting that $x$ is a fixed parameter, we may write
\begin{multline}
\mathscr{Q}_{t_n}\!\otimes\!\mathscr{P}_{t_n}\!\!\left(\frac{-h(t_n,2^{1/3}t_n^{2/3}x)+\inv{24}t_n}{2^{-1/3}t_n^{1/3}}\leq r\right)\\
=\mathscr{Q}_{t_n}\!\otimes\!\mathscr{P}_{t_n}\!\!\left(2^{1/3}t_n^{-1/3}\log\int d\mu^\delta_{t_n}(y)\, e^{2^{-1/3}t_n^{1/3}\big[\mathcal{A}^{t_n}_{\delta,x}(x-y)-(x-y)^2\big] }\leq r\right)
\end{multline}
where $\mathcal{A}^t_{\delta,x}(z)=\mathcal{A}^t(z)+\delta(z-x)^2$.
Now $\mu^\delta_{t}$ satisfies a large deviation principle with rate function $\mathfrak{h}_0(x)+\delta x^2$, while $\mathcal{A}^{t_n}_{\delta,x}(x-y)$ converges in distribution to $\tilaip(x-y)+\delta y^2$.
On the other hand, if we let $\bar\delta_1>0$ be such that \eqref{qbd} holds for $I(y)=\mathfrak{h}_0(y)$, then for $I(y)=\mathfrak{h}_0(y)+\delta y^2$ the same assumption holds now with $\delta_1=\bar\delta_1+\delta$.
By \eqref{eq:pvBd} we see then that $\mathcal{A}^t_{\delta,x}$ satisfies \eqref{abd} with some choice of $\delta_2<\delta_1$. 
By Lemma \ref{lem:lap} we now deduce that the above probability converges to
\begin{equation}
  \mathscr{Q}\!\otimes\!\widetilde{\mathscr{P}}\!\left(\sup_{y\in\rr}\left\{ \tilaip(x-y)- (x-y)^2 - \mathfrak{h}_0(y)\right\}\leq r\right)
\end{equation}
as desired.

It remains to prove \eqref{eq:pvBd}, but this follows straightforwardly from the following result from \cite{corwinHammondKPZ}:

\begin{lem}[{\cite[Lemma 4.1]{corwinHammondKPZ}}]
  Let\footnote{We point out that in \cite{corwinHammondKPZ} the definition of $\mfHnw$ has no minus sign in front of $\hnw$; this just reflects the fact that, compared to them, we chose the opposite sign of the non-linearity in \eqref{eq:KPZ}.}
  \[\mfHnw_t(x)=\frac{-\hnw(t,t^{2/3}x)+\frac1{24}t}{t^{1/3}}\]
  and fix $\alpha>0$ and $\eta>0$.
  Then there is a constant $\beta>0$ such that, for all $t\geq1$,
  \[\pp\!\left(\mfHnw_t(y)+\tfrac12y^2\geq\alpha y^2+\beta~~\forall\,y\in\rr\right)<\eta.\]
\end{lem}

The context in which they are able to obtain this estimate is their construction of the \emph{KPZ line ensemble}, which allows them to view the solution of the KPZ equation with narrow wedge initial data as the top line in an infinite collection of curves which enjoy what they call the $\mathbf{H}_t$-Brownian Gibbs property.
They prove the lemma, among many other results about the solution of the KPZ equation, by carefully making use of this construction.
The limit \eqref{eq:pvBd} follows directly from the lemma after observing that $\mfHnw_t(x)=2^{-1/3}\mathcal{A}^t(2^{1/3}x)-\tfrac12x^2$.

\subsection{Proof of Corollary \ref{thm:main}}

We start by noting that $\mathfrak{h}_0(x)+\delta x^2$ satisfies \eqref{qbd} for all three choices of $\mathfrak{h}_0$ (i.e. $\mathfrak{h}_0(0)=0$ and $\mathfrak{h}_0(x)=-\infty$ for $x\neq0$; $\mathfrak{h}_0\equiv0$; and $\mathfrak{h}_0(x)=\sqrt{2}B(x)$) and any $\delta\in[0,1)$, and thus Theorem \ref{thm:lapl-KPZ} can be applied to obtain tightness of the $F^\mu_{t,x}$.
So we only need to prove the remaining claims.

Consider first the curved case.
From \cite{acq} it is known that if the initial data is prescribed as $Z(0,x)=\delta_0(x)$ for the SHE, then $F^\mu_{t,x}(r)$ actually converges to $F_{\rm GUE}(r+x^2)$.
The limit in \eqref{eq:limCurved} then follows from Corollary \ref{cor:domain}, since this choice of initial data clearly belongs to the domain of attraction characterized by our choice of $\mathfrak{h}_0$.

Now suppose that $\mu$ satisfies \eqref{eq:muCurved}.
We need to show that the measures $\mu^0_t$ satisfy a large deviation principle with rate function $\mathfrak{h}_0(0)=0$  and $\mathfrak{h}_0(x)= \infty$ if $x\neq0$.
For closed $\mathcal{C}\subset\rr$ this translates into showing that
\[\limsup_{t\to \infty}t^{-1/3}\log(\mu^0_t(\mathcal{C}))=-\infty\;\;\; \text{if }0\notin \mathcal{C}
\qand \limsup_{t\to \infty}t^{-1/3}\log(\mu^0_t(\mathcal{C}))\leq 0\;\;\; \text{if }0\in\mathcal{C}.\]
If $0\notin \mathcal{C}$ then $\mathcal{C}\subset(-r,r)^\text{c}$ for some $r>0$, and then by hypothesis
\begin{equation}\label{eq:closedCurved}
  t^{-1/3}\log(\mu^0_t(\mathcal{C}))\leq t^{-1/3}\big[\log(c)-\kappa (2^{-1/3}t^{2/3}r)^{1/2+\delta}\big]\xrightarrow[t\to \infty]{}- \infty.
\end{equation}
The case $0\in\mathcal{C}$ is even simpler because by hypothesis we know that $\mu^0_t(\rr)=\mu(\rr)< \infty$, and thus 
\[t^{-1/3}\log(\mu^0_t(\mathcal{C}))\leq t^{-1/3}\log(\mu(\rr))\xrightarrow[t\to \infty]{}0.\]
On the other hand, in the case of open $\mathcal{O}\subset\rr$ we only need to show that
\[\liminf_{t\to \infty}t^{-1/3}\log(\mu^0_t(\mathcal{O}))\geq0\quad\text{if }0\in \mathcal{O}.\]
But if $0\in\mathcal{O}$, then $\mathcal{O}\supseteq(-r,r)$ for some $r>0$, so $\mu^0_t(\mathcal{O})\geq\mu([-2^{-1/3}t^{1/3}r,2^{-1/3}t^{1/3}r])$.
Since $\mu(\rr)>0$, we may find some $M>0$ so that the last quantity is bounded below by $M$ for large enough $t$, and thus the estimate follows.

We turn now to the flat case.
The fact that the ordering holds is just \eqref{eq:orderGOEGUE}.
Now suppose that $\mu$ satisfies \eqref{eq:muFlat}. 
We will show that for any $0<\delta<1$ the measures $\mu^\delta_t$ satisfy a large deviation principle with rate function $\mathfrak{h}_0(x)=\delta x^2$.

Consider a closed $\mathcal{C}\subseteq\rr$.
If $\mathcal{C}=\emptyset$ the condition is trivial.
Otherwise, assume first that $0\notin\mathcal{C}$ and let $r=\min\{|x|\!:x\in\mathcal{C}\}$, which is necessarily positive (and finite).
Since $\mathcal{C}\subseteq(-r,r)^\text{c}$, using \eqref{eq:deltaResc} and the first inequality in \eqref{eq:muFlat} and integrating by parts we get
\begin{equation}\label{eq:muFlatClosed}
    \mu_t^\delta(\mathcal{C}\cap[0,\infty))
    \leq\int_{x>r} d(\mu_t([r,x]))\,e^{-\delta2^{-1/3}t^{1/3}x^2}
    \leq c\tsm\int_{x>r}\!\!\!dx\, t^{1/3}xe^{-\delta2^{-1/3}t^{1/3}x^2+\kappa_2 (t^{2/3}x)^{1/2-\delta_2}}
 \end{equation}
for some constant $c>0$.
The same bound can be obtained for $\mu_t^\delta(\mathcal{C}\cap(-\infty,0])$, and thus we get $\log(\mu_t^\delta(\mathcal{C}))\leq c'-\delta2^{-1/3}t^{1/3}r^2+\kappa_2t^{1/3-\delta_2/3}r^{1/2-\delta_2}$ for some $c'>0$.
Hence
\[\limsup_{t\to \infty}2^{1/3}t^{-1/3}\log(\mu^\delta_t(\mathcal{C}))\leq-\delta r^2=-\inf_{x\in\mathcal{C}}\delta x^2\]
by our choice of $r$, which gives the desired condition.
The case $0\in\mathcal{C}$ follows from a similar, and simpler, argument, bounding $\mu_t^\delta(\mathcal{C})$ directly by $\mu_t^\delta(\rr)$.

Now take an open $\mathcal{O}\subseteq\rr$.
As before, if $\mathcal{O}=\emptyset$ the inequality is trivial.
Otherwise, let $r=\inf\{|x|\!:x\in\mathcal{O}\}\in[0,\infty)$.
We will assume for simplicity that $r\in\overline{\mathcal{O}}$ (otherwise $-r\in\overline{\mathcal{O}}$ and the same argument works).
Then there is some $\ep>0$ such that $(r,r+\ep)\subseteq\mathcal{O}$, and then by an argument similar to \eqref{eq:muFlatClosed}, but using now the second inequality in \eqref{eq:muFlat}, we have
\begin{align}
  \mu_t^\delta(\mathcal{O})&\geq\int_{r}^{r+\ep}d(-\mu_t([x,r+\ep]))\,e^{-\delta2^{-1/3} t^{1/3}x^2}\\
  &=\int_{r}^{r+\ep}\!\!dx\,2^{2/3}\delta t^{1/3}xe^{-\delta2^{-1/3} t^{1/3}x^2}\mu_t([x,r+\ep])+e^{-\delta2^{-1/3} t^{1/3}r^2}\mu_t([r,r+\ep])\\
  &\geq c_1\ts e^{-\delta2^{-1/3} t^{1/3}r^2-\kappa_1(t^{2/3}r)^{1/2-\delta_1}}.
\end{align}
From this we get 
\[\limsup_{t\to \infty}2^{1/3}t^{-1/3}\log(\mu^\delta_t(\mathcal{O}))\geq-\delta r^2=-\inf_{x\in\mathcal{O}}\delta x^2,\]
which gives the desired condition.

Finally, in the stationary case, the proof of \eqref{eq:limStat} is the same as the one given for \eqref{eq:limCurved}, using the result of \cite{borCorFerrVeto} (which computed the limiting one-point fluctuations of the KPZ equation with initial condition given as $Z(0,x)=e^{-B(x)}$ for the SHE) in place of the result of \cite{acq}.
The statement about $\mu(dx)=e^{\Theta(x)}dx$ with $\Theta$ converging to a Brownian motion under diffusive scaling is easy to check (and is the same as Example \ref{ex:inFn}).

\section{Hitting probabilities for the Airy$_2$ process}\label{sec:hitting}

The goal of this section is to prove Theorem \ref{thm:hittProb}. 
Our starting point is Theorem 2 in \cite{cqr}, which gives a formula for the probability that
the Airy$_2$ process hits a general function on a finite interval. Before stating the
formulas, let us recall the definition of the Airy$_2$ process and introduce some notation
(for more details see \cite{quastelRem-review}).

First recall the definition \eqref{eq:Pm} of $P_m$ as the projection onto the interval $(m,\infty)$ and if $\bar P_m=I-P_m$.
Next define the \emph{Airy Hamiltonian} 
\begin{equation}\label{eq:airyHam}
	H=-\Delta+x.
\end{equation}
Since the Airy function satisfies $\Ai''(x)=x\Ai(x)$, we have $H\!\hspace{0.05em}\Ai(x-\lambda)=\lambda\!\hspace{0.1em}\Ai(x-\lambda)$.
Hence the shifted Airy functions are the (generalized) eigenfunctions of the Airy Hamiltonian $H$, and the Airy kernel $\K$ (defined in \eqref{eq:KAi}) is the projection onto the negative eigenspace associated to $H$.
In particular, this implies that, even though $e^{tH}$ is not well-defined for
$t>0$, $e^{tH}\K$ makes sense for all $t\in\rr$ and its integral kernel is given by
\[e^{tH}\K(x,y)=\int_{-\infty}^0 d\lambda\,e^{\lambda t}\Ai(x-\lambda)\Ai(y-\lambda).\]
Making use of these objects, the Airy$_2$ process can be defined through its finite-dimensional distributions by the following determinantal formula: given
$x_1,\dots,x_m\in\mathbb{R}$ and $t_1<\dots<t_m$ in $\mathbb{R}$,
\begin{multline}\label{eq:airy2PrDef}
  \pp\!\left(\aip(t_1)\le x_1,\dots,\aip(t_m)\le x_m\right) \\
  = \det\!\left(I-K_{\Ai}+\bar P_{x_1}e^{(t_1-t_2)H}\bar P_{x_2}e^{(t_2-t_3)H}\dotsm
   \bar P_{x_n}e^{(t_n-t_1)H}K_{\Ai}\right),
\end{multline}
where, as it is the case throughout this paper, the Fredholm determinant is computed in $L^2(\rr)$.
We remark that this is not the most usual definition of the finite-dimensional distributions of the Airy$_2$ process, which utilizes the so-called extended Airy kernel and a Fredholm determinant computed on $L^2(\{t_1,\dotsc,t_n\}\times\rr)$ (see \cite{prahoferSpohn,prolhacSpohn,quastelRem-review,bcr} for more on this), but it is the one that is most convenient for our purposes.
In fact, this formula was used in \cite{cqr} to derive a formula for the probability that the Airy$_2$ process stays below a function $g$ on a finite interval $[\ell_1,\ell_2]$ by essentially discretizing $g$ on a finite mesh of $[\ell_1,\ell_2]$, using \eqref{eq:airy2PrDef}, and then making the mesh size go to 0. 
Before stating the formula we need the following definition.

\begin{defn}\label{defn:theta}
  Fix $\ell_1<\ell_2$ and let $g\in H^1([\ell_1,\ell_2])$. We define
  $\Theta^g_{\ell_1,\ell_2}$ as the solution operator (acting on $L^2(\rr)$) of the
  following boundary value problem:
  \begin{equation}
    \begin{aligned}
      \p_tu+Hu&=0\quad\text{for }x<g(t)+t^2, \,\,t\in (\ell_1,\ell_2)\\
      u(\ell_1,x)&=f(x)\uno{x<g(\ell_1)+\ell_1^2}\\
      u(t,x)&=0\quad\text{for }x\ge g(t)+t^2.
    \end{aligned}\label{eq:bdval}
  \end{equation}
  This means that for $f\in L^2(\rr)$, $\Theta^g_{\ell_1,\ell_2}f(\cdot)=u(\ell_2,\cdot)$,
  where $u(\ell_2,\cdot)$ is the solution at time $\ell_2$ of \eqref{eq:bdval}.
\end{defn}
 
We remark that this definition of $\Theta^g_{\ell_1,\ell_2}$ is slightly different from
the one in \cite{cqr} because for convenience here we have shifted $g(t)$ by $t^2$ in
\eqref{eq:bdval}.

The continuum statistics formula derived in \cite{cqr} can then be written as follows: for
$g\in H^1([\ell_1,\ell_2])$,
 \begin{equation}
   \pp\!\left(\aip(t)\leq g(t)+t^2\text{ for }t\in[\ell_1,\ell_2]\right)
   =\det\!\left(I-K_{\Ai}+e^{-\ell_1 H}\K\Theta^g_{\ell_1,\ell_2}  e^{\ell_2H}K_{\Ai}\right).\label{eq:cqr}
 \end{equation}
To obtain this formula from Theorem 2 of \cite{cqr} we have used the cyclic property of the determinant,
\[\det(I+K_1K_2)=\det(I+K_2K_1)\]
whenever $K_1$ and $K_2$ are operators on $L^2(\rr)$ such that $K_1K_2$ and $K_2K_1$ are trace class, together with the fact that $e^{(s_1+s_2)H}\K=e^{s_1 H}\K e^{s_2H}\K$ for all $s_1,s_2\in\rr$.
For more details see \cite{cqr} or \cite{quastelRem-review} (particularly Sections 1.5.2 and 3).

In view of \eqref{eq:cqr} we need to take $-\ell_1=\ell_2=L$ and compute
\[\lim_{L\to\infty}e^{LH}\K\Theta^g_{-L,L}e^{LH}K_{\Ai}.\]
We start by noting that the operator $\Theta^g_{\ell_1,\ell_2}$ can be expressed as an integral operator on $L^2(\rr)$ with an explicit integral kernel, which is given in Theorem 3 of \cite{cqr}:
\begin{multline}\label{eq:ThetaL}
 \Theta^g_{\ell_1,\ell_2}(x_1,x_2)=e^{\ell_2^3/3-\ell_1^3/3+\ell_1
   x_1-\ell_2x_2}p(\ell_2-\ell_1,x_2-x_1)\\
 \times\pp_{\ell_1,x_1-\ell_1^2;\ell_2,x_2-\ell_2^2}\!\left(B(t)\leq g(t)\text{ on
   }[\ell_1,\ell_2]\right),
\end{multline}
where the probability is computed with respect to a Brownian bridge $B(s)$ from
$x_1-\ell_1^2$ at time $\ell_1$ to $x_2-\ell_2^2$ at time $\ell_2$ and with diffusion
coefficient $2$, and where $p(t,x)=(4\pi t)^{-1/2}e^{-x^2/4t}$ denotes the transition
kernel of a Brownian motion with diffusion coefficient 2.

Now define
\begin{equation}
 \label{eq:barScatL}
 \overline{\scat}^{g}_L=e^{LH}\K\Theta^g_{-L,L}e^{LH}\K.
\end{equation}

\begin{thm}\label{thm:limL}
  Suppose that $g\in H^1_{\rm ext}(\rr)$ is such that $g(t)\geq c-\kappa\ts t^2$ for some $\kappa\in[0,3/4)$.
  Then the family of operators $\big(\K-\overline{\scat}^{g}_L\big)_{L>0}$ has a limit in trace class norm.
  Moreover, the limit has the following explicit representation: 
  \[\lim_{L\to\infty}\big(\K-\overline{\scat}^{g}_L\big)=\K-A\bar P_{0}\scat^{g}\bar P_{0}A^*,\]
  where $\scat^{g}$ is the Brownian scattering operator introduced in Definition \ref{defn:brScatt}.
\end{thm}

The fact that $\K-A\bar P_0\scat^{g}\bar P_0A^*$ is a trace class operator mapping
$L^2(\rr)$ to itself is proved in Proposition \ref{prop:OmegaInfty-TrCl} below.

Theorem \ref{thm:hittProb} is a straightforward consequence of Theorem \ref{thm:limL} together
with \eqref{eq:cqr} and the continuity of the Fredholm determinant in the trace class
norm. The rest of this section is thus devoted to the proof of Theorem \ref{thm:limL}.

Let us first proceed formally, using the Baker-Campbell-Hausdorff formula to rewrite
$e^{tH}$ in order to find a more convenient expression for $\overline{\scat}^g_L$. Slightly more
generally, we will work with the operator $e^{-\ell_1H}\K\Theta^g_{\ell_1,\ell_2}e^{\ell_2 H}K_{\Ai}$ for $\ell_1<\ell_2$. Let $e^{t\xi}$
($\xi$ stands for a generic variable) denote the multiplication operator
$(e^{t\xi}f)(x)=e^{tx}f(x)$. Then by the Baker-Campbell-Hausdorff formula we may
(formally) write
\begin{equation}\label{eq:BCH}
  e^{tH}=e^{-t^3/3+t\xi}e^{-t\Delta-t^2\nabla}=e^{-t\Delta+t^2\nabla}e^{-t^3/3+t\xi}
\end{equation} (see the justification of (4.10) in \cite{quastelRem-review} for the formal
computation). Now recalling that $e^{a\nabla}$ acts as a shift
(i.e. $e^{a\nabla}f(x)=f(x+a)$)  we get
\begin{equation}\label{eq:BCHed}
    e^{-\ell_1 H}\Theta^g_{\ell_1,\ell_2}e^{\ell_2H}
    =e^{-\ell_1\Delta}e^{\ell_1^2\nabla}e^{\ell_1^3/3-\ell_1\xi}
    \Theta^{g}_{\ell_1,\ell_2}e^{-\ell_2^3/3+\ell_2\xi}e^{-\ell_2^2\nabla}e^{-\ell_2\Delta}
    =e^{\ell_1\Delta}\wt\Theta^g_{\ell_1,\ell_2}e^{-\ell_2\Delta}
\end{equation}
(where in the last equality we used $e^{a\xi}e^{b\nabla}=e^{-ab}e^{b\nabla}e^{a\xi}$) with
\begin{equation}
  \wt\Theta^{g}_{\ell_1,\ell_2}(x_1,x_2)=p(\ell_2-\ell_1,x_2-x_1+\ell_2^2-\ell_1^2)
  \pp_{\ell_1,x_1;\ell_2,x_2}\!\left(B(t)\leq g(t)\text{ on }[\ell_1,\ell_2]\right).\label{eq:wtTheta}
\end{equation}
To make sense of this formula we need to understand what it means to write $e^{-t\Delta}$
for $t>0$ in our case.
The key point is that we are allowed to consider the operator $e^{-t\Delta}$ for $t>0$ as long as it is applied to $A$ or, in other words, that the function $x\longmapsto\Ai(x)$ lies in the image of the heat kernel.
More generally, the same fact holds for $e^{-t\Delta}$ applied to $e^{\alpha\xi}A$ for any $\alpha\in \rr$, as required in \eqref{eq:tildePsi}.
This fact is well-known in the literature in the case $\alpha=0$, see e.g. \cite[Prop. 1.2]{quastelRemAiry1}.

\begin{prop}\label{prop:mLaplacian}
  Given $\alpha,t\in\rr$ define $\varphi^\alpha_{t}(x)=e^{-2t^3/3+2\alpha t^2-\alpha^2t+x(\alpha-t)}\Ai(x+t^2-2\alpha t)$.  Then for all $s,t>0$ we have $e^{s\Delta}\varphi^\alpha_{t}(x)=\varphi^\alpha_{t-s}(x)$.
  In particular, $e^{t\Delta}\varphi^\alpha_{t}(x)=e^{\alpha x}\tsm\Ai(x)$, and as a consequence the kernels $e^{t\Delta}e^{\alpha\xi}A$ and $A^*e^{\alpha\xi}e^{t\Delta}$ are well defined for every $\alpha,t\in\rr$ via the formula
  \begin{equation}
    e^{t\Delta}e^{\alpha\xi}A(x,\lambda)=A^*e^{\alpha\xi}e^{t\Delta}(\lambda,x)
    =e^{2t^3/3+2\alpha t^2+\alpha^2t+x(t+\alpha)-\lambda t}\Ai(x-\lambda+t^2+2\alpha t)\label{eq:e-tDB0}
  \end{equation}
  and they satisfy the semigroup property in the sense that
  $e^{(s+t)\Delta}e^{\alpha\xi}A=e^{s\Delta}e^{t\Delta}e^{\alpha\xi}A$ and $A^*e^{\alpha\xi}e^{(s+t)\Delta}=A^*e^{\alpha\xi}e^{s\Delta}e^{t\Delta}$ for all $s,t\in\rr$.
\end{prop}

\begin{proof}
	The proof of the formula $e^{s\Delta}\varphi^\alpha_{t}(x)=\varphi^\alpha_{t-s}(x)$ is completely analogous to the one of Proposition 1.2 in \cite{quastelRemAiry1}, so we omit it.
	To get \eqref{eq:e-tDB0}, note that $e^{\alpha\xi}Ae^{-\alpha\xi}(x,\lambda)=\varphi^\alpha_0(x-\lambda)$, so $e^{t\Delta}e^{\alpha\xi}A(x,\lambda)=\varphi^\alpha_{-t}(x-\lambda)e^{\alpha\lambda}$, whence the formula follows.
\end{proof}

As a consequence of this and \eqref{eq:K-AiryTr}, and since $e^{tH}$ formally commutes with $\K$, we may rewrite \eqref{eq:BCHed} as
\begin{equation}\label{eq:wBCHed}
  e^{-\ell_1 H}\K\Theta^{g}_{\ell_1,\ell_2}e^{\ell_2H}K_{\Ai}
  =A\bar P_0(A^*e^{\ell_1\Delta})\wt\Theta^{g}_{\ell_1,\ell_2}(e^{-\ell_2 \Delta}A)\bar P_0A^*.
\end{equation}
The derivation of this identity has been non-rigorous, but the main point is that the
right hand side of \eqref{eq:wBCHed} is now well defined thanks to Proposition
\ref{prop:mLaplacian}. The identity can be justified rigorously, without appealing to the
Baker-Campbell-Hausdorff formula, by relatively simple integral calculations. We will skip
the details of that justification and instead refer the reader to \cite{qr-airy1to2} where
the argument is carried out in detail for a similar problem (see in particular the proof
of Lemma 2.1 there).

Next we rewrite the probability appearing in \eqref{eq:wtTheta} as follows: assuming $\alpha\in(\ell_1,\ell_2)$
\begin{multline}
  \pp_{\ell_1, x_1;\ell_2, x_2}\!\left(B(t)\leq g(t)\text{ on }[\ell_1,\ell_2]\right)
  = \int_{-\infty}^{g(\alpha)}\pp_{\ell_1, x_1; \ell_2, x_2} ( B(\alpha)\in dy)\\
  \times \pp_{\ell_1, x_1; \alpha, y} (B(t)<g(t)\text{ on }[\ell_1,\alpha])\pp_{\alpha, y;\ell_2, x_2}(
  B(t)<g(t)\text{ on } [\alpha,\ell_2]).
\end{multline}
Recalling the definitions of $g^+$ and $\tau_g$ in \eqref{eq:defTaug} and \eqref{eq:gplusminus} and of $g_\alpha(t)=g(t+\alpha)$, the last probability in the above integral can be rewritten as
\begin{equation}
  \pp_{0, y;\ell_2-\alpha, x_2}(B(t)<g_\alpha(t)\text{ on } [0,\ell_2-\alpha])
  =1-\int_0^{\ell_2-\alpha}\pp_y(\tau_{g_\alpha^+}\in dt)\tfrac{p(\ell_2-\alpha-t,x_2-g_\alpha(t))}{p(\ell_2-\alpha,x_2-y)}.
\end{equation}
A similar identity can be written for $\pp_{\ell_1, x_1; \alpha, y}(B(t)<g(t)\text{ on }[\ell_1,\alpha])$, now using $\tau_{g_a^-}$, which we take to be independent of $\tau_{g_\alpha^+}$, and going backwards from time $\alpha$ to time
$\ell_1$.
Using this and writing $\pp_{\ell_1, x_1; \ell_2, x_2}(B(\alpha)\in dy)$ explicitly we find that
\begin{align}
  \pp_{\ell_1, x_1;\ell_2, x_2}\!\left(B(t)<g(t)\text{ on }[\ell_1,\ell_2]\right)
  &=\int_{-\infty}^{g(\alpha)}dy\sqrt{\tfrac{\ell_2-\ell_1}{4\pi(a-\ell_1)(\ell_2-\alpha)}}
  e^{-\frac{((\ell_2-\alpha)x_1+(\alpha-\ell_1)x_2+(\ell_1-\ell_2)y)^2}{4(\alpha-\ell_1)(\ell_2-\alpha)(\ell_2-\ell_1)}}\\
  &\hspace{0.35in}\times\left(1-\int_0^{\alpha-\ell_1}\pp_y(\tau_{g_\alpha^-}\in dt_1)\tfrac{p(
      \alpha-\ell_1 - t_1, x_1-g_\alpha(-t_1))}{ p(\alpha-\ell_1 , x_1-y) }\right)\\
  &\hspace{0.35in}\times\left(1-\int_{0}^{\ell_2-\alpha}\pp_y(\tau_{g_\alpha^+}\in dt_2)\tfrac{
      p(\ell_2-\alpha - t_2, x_2-g_\alpha(t_2))}{p(\ell_2-\alpha , x_2-y) } \right)
\end{align}
Recalling from \eqref{eq:wtTheta} that in the formula for
$\wt\Theta^g_{\ell_1,\ell_2}(x_2-x_1)$ this probability is premultiplied by
$p(\ell_2-\ell_1,x_2-x_1+\ell_2^2-\ell_1^2)$ and observing that
\[\tfrac{p(\ell_2-\ell_1,x_2-x_1+\ell_2^2-\ell_1^2)}{p(a-\ell_1,x_1-y)p(\ell_2-a,x_2-y)}\sqrt{\tfrac{\ell_2-\ell_1}{4\pi(a-\ell_1)(\ell_2-a)}}
e^{-\frac{((\ell_2-a)x_1+(a-\ell_1)x_2+(\ell_1-\ell_2)y)^2}{4(a-\ell_1)(\ell_2-a)(\ell_2-\ell_1)}}
=e^{\frac14(\ell_1^2-\ell_2^2+2x_1-2x_2)(\ell_1+\ell_2)}\] we deduce that
\begin{multline}
  \wt\Theta^g_{\ell_1,\ell_2}(x_1,x_2)=e^{\frac14(\ell_1^2-\ell_2^2+2x_1-2x_2)(\ell_1+\ell_2)}\\
  \times\int_{-\infty}^{g(\alpha)}dy\left[p(\alpha-\ell_1 , x_1-y)-\int_0^{\alpha-\ell_1}dt_1\,\pp_{y}(\tau_{g_\alpha^-}\in dt_1)p(\alpha-\ell_1 - t_1, x_1-g_\alpha(-t_1))\right]\\
  \times\left[p(\ell_2-\alpha , x_2-y)-\int_{0}^{\ell_2-\alpha}dt_2\,\pp_{y}(\tau_{g_\alpha^+}\in
    dt_2)p(\ell_2 -\alpha- t_2, x_2-g_\alpha(t_2))\right].
\end{multline}
At this point we go back to the case $-\ell_1=\ell_2=L$, for which the last identity yields the following (note that the exponential prefactor above vanishes in this case):
\begin{prop}\label{prop:wtTheta2}
  For any $\alpha\in(-L,L)$ we have
  \begin{multline}
    A^*e^{-L\Delta}\wt\Theta^g_{-L,L}e^{-L\Delta}A(\lambda_1,\lambda_2)\\
    =\int_{-\infty}^{g(\alpha)}dy\left[A^*e^{\alpha \Delta}(\lambda_1,y)-\int_{0}^{\alpha+L}dt_1\,\pp_y(\tau_{g_\alpha^-}\in
      dt_1)A^*e^{(\alpha-t_1)\Delta}(\lambda_1,g_\alpha(-t_1))\right]\\\times
    \left[e^{-\alpha \Delta}A(y,\lambda_2)-\int_{0}^{L-\alpha}dt_2\,\pp_y(\tau_{g_\alpha^+}\in
      dt_2)e^{-(\alpha+t_2)\Delta}A(g_\alpha(t_2),\lambda_2)\right].
  \end{multline}
\end{prop}

Let us rewrite the last identity as follows:
\begin{equation}
  \label{eq:rewr-wtTheta2}
  A^*e^{-L\Delta}\wt\Theta^g_{-L,L}e^{-L\Delta}A=\Psi^{g_\alpha^-,-\alpha}_L\bar P_{g(\alpha)}(\Psi^{g_\alpha^+,\alpha}_L)^*
\end{equation}
with
\begin{equation}
  \label{eq:rewr-wtTheta2-Omega}
    \Psi^{g,\alpha}_L(\lambda,y)=e^{-\alpha \Delta}A^*(\lambda,y)-\int_{0}^{L-\alpha}dt\,\pp_y(\tau_{g}\in
      dt)A^*e^{-(\alpha+t)\Delta}(\lambda,g(t))
\end{equation}
The pointwise limit as $L\to\infty$ of \eqref{eq:rewr-wtTheta2-Omega} is straightforward to compute, and gives $\Psi^g_\alpha$.
Therefore, in view of \eqref{eq:defScatgAlpha}, the right hand side of \eqref{eq:rewr-wtTheta2} converges pointwise to $\scat^{g}$.
In particular, since the left hand side of \eqref{eq:rewr-wtTheta2} does not depend on $\alpha$, we have proved the following

\begin{prop}\label{prop:alphaIndep}
The kernel $\Psi^{g^-_\alpha}_{-\alpha}\ts\bar P_{g(\alpha)}\ts(\Psi^{g^+_\alpha}_\alpha)^*$ does not depend on $\alpha$.
\end{prop}

This shows that, as claimed, the definition of $\scat^g$ through \eqref{eq:defScatgAlpha} does not depend on $\alpha$.

Our next goal is to upgrade the pointwise convergence of the right hand side of \eqref{eq:rewr-wtTheta2} to $\scat^g$ to trace class convergence of $\overline{\scat}^{g}_L$ (under suitable conditions on $g$), which by \eqref{eq:barScatL}, \eqref{eq:wBCHed} and \eqref{eq:rewr-wtTheta2} satisfies
\[\overline{\scat}^{g}_L=A\bar P_{0}\Psi^{g_\alpha^-,-\alpha}_L\bar P_{g(\alpha)}(\Psi^{g_\alpha^+,\alpha}_L)^*\bar P_{0}A^*.\]
We will first prove the following result, which already contains all the necessary ideas.

\begin{prop}\label{prop:OmegaInfty-TrCl}
  For any $g\in H^1_{\rm ext}(\rr)$ such that $g(t)\geq c-\kappa t^2$ for some $c\in\rr$ and $\kappa\in(0,\frac34)$, $\K-A\bar P_{0}\scat^{g}\bar P_{0}A^*$ is a trace class operator.
\end{prop}

\begin{proof}
  Choose $\alpha$ so that $g(\alpha)<\infty$ (the result is trivial if $g\equiv\infty$) and write $\Psi^g_\alpha$ as 
  \[\Psi^g_\alpha=e^{-\alpha \Delta}A^*-\Phi^g_\alpha.\]
  Then, since $\K=A\bar P_{0}A^*$,
  \begin{multline}\label{eq:KOmInfty}
	A\bar P_{0}\scat^{g}\bar P_{0}A^*-\K
	=\left[A\bar P_{0}A^*e^{\alpha \Delta}\bar P_{g(\alpha)}e^{-\alpha \Delta}A\bar P_{0}A^*-A\bar P_{0}A^*\right]\\
	-A\bar P_{0}A^*e^{\alpha \Delta}\bar P_{g(\alpha)}(\Phi^{g_\alpha^+}_\alpha)^*\bar P_{0}A^*
	-A\bar P_{0}\Phi^{g_\alpha^-}_{-\alpha}\bar P_{g(\alpha)}e^{-\alpha \Delta}A\bar P_{0}A^*\\
	+A\bar P_{0}\Phi^{g_\alpha^-}_{-\alpha}\bar P_{g(\alpha)}(\Phi^{g_\alpha^+}_\alpha)^*\bar P_{0}A^*.
  \end{multline}
  Since $A\bar P_{0}A^*=A\bar P_{0}A^*e^{\alpha \Delta}e^{-\alpha \Delta}A\bar P_{0}A^*$, the first operator on the right hand side equals $-A\bar P_{0}A^*e^{\alpha \Delta}P_{g(\alpha)}e^{-\alpha \Delta}A\bar P_{0}A^*$, and therefore it is trace class because $A\bar P_0A^*e^{\alpha \Delta}P_m$ is Hilbert-Schmidt for any $\alpha,m\in\rr$:
  using the fact that
  \begin{equation}\label{eq:airyBd}
    |\!\Ai(x)|\leq c\,e^{-\frac23x^{3/2}}\text{ for }x\geq0,\qquad |\!\Ai(x)|\leq c\text{ for }x<0
  \end{equation}
  for some $c>0$, we have (using $AA^*=I$ again)
  \begin{equation}\label{eq:aboveas}
    \begin{split}
      &\|A\bar P_0A^*e^{\alpha \Delta}P_m\|^2_2=\int_{-\infty}^\infty dx_1\int_m^\infty dx_2\,[A\bar P_0A^*e^{\alpha \Delta}(x_1,x_2)]^2\\
      &\quad~~=\int_{-\infty}^\infty \!\!dx_1\int_m^\infty \!\!dx_2\int_{-\infty}^0\!\! d\lambda_1\int_{-\infty}^0 \!\!d\lambda_2\Ai(x_1-\lambda_1)e^{2\alpha^3/3+\alpha(x_2-\lambda_1)}\tsm\Ai(x_2-\lambda_1+\alpha^2)\\
      &\hspace{2in}\times\Ai(x_1-\lambda_2)e^{2\alpha^3/3+\alpha(x_2-\lambda_2)}\tsm\Ai(x_2-\lambda_2+\alpha^2)\\
      &\quad~~=\int_m^\infty dx_2\int_{-\infty}^{0} d\lambda_2\, e^{4\alpha^3/3+2\alpha(x_2-\lambda_2)}\tsm\Ai(x_2-\lambda_2)^2<\infty.
    \end{split}
  \end{equation}
  Hence we need to show that each of the remaining three operators on the right hand side of
  \eqref{eq:KOmInfty} are trace class.

  Fix $\beta>|\alpha|$. We will focus on the third term on the right hand side of \eqref{eq:KOmInfty}, the remaining two can be handled in exactly the same way (the fourth one is slightly simpler because $\beta$ can be taken to be 0). 
  We write this operator as $L_1L_2$ with
  \[L_1=A\bar P_0\Phi^{g_\alpha^-}_{-\alpha}e^{-\beta\xi}\bar P_{g(\alpha)}
  \qqand L_2=\bar P_{g(\alpha)}e^{\beta\xi}e^{-\alpha \Delta}A\bar P_0A^*\]
  so that it suffices to see that both $L_1$ and $L_2$ are Hilbert-Schmidt.
  To check this for the second factor we proceed as in \eqref{eq:aboveas}:
  \begin{multline}
  \|L_2\|^2_2
    =\int_{-\infty}^{g(\alpha)}dx\int_{-\infty}^0\!\!d\lambda\,e^{2\beta x-4\alpha^3/3-2\alpha(x-\lambda)}\tsm\Ai(x-\lambda+\alpha^2)^2\\
    \leq\int_{-\infty}^0\!\!d\lambda\,e^{2\beta \lambda}\int_{-\infty}^\infty dx\,e^{2(\beta-\alpha)x}\Ai(x+\alpha^2)^2,
  \end{multline}
  which is finite by \eqref{eq:airyBd} and the fact that $\beta>\alpha$.
  On the other hand, $\|L_1\|^2_2$ equals
  \begin{equation}
	  \begin{split}
      \int_{-\infty}^{g(\alpha)}dx&\int_{-\infty}^0\!\!d\lambda
      \Bigg[\int_0^\infty\pp_x(\tau_{g_\alpha^-}\in dt)e^{-\frac23(\alpha+t)^3-(g^-_\alpha(t)-\lambda)(\alpha+t)-\beta x}\Ai(g^-_\alpha(t)-\lambda+(\alpha+t)^2)\Bigg]^2\\
      &=\int_{-\infty}^{g(\alpha)}dx \int_{-\infty}^0\!\!d\lambda \Bigg[\int_0^\infty dt\,\pp_x(\tau_{g_\alpha^-}\leq t)\\
      &\hspace{1.3in}\times\p_t\big(e^{-\frac23(\alpha+t)^3-(g^-_\alpha(t)-\lambda)(\alpha+t)-\beta x}\Ai(g^-_\alpha(t)-\lambda+(\alpha+t)^2)\big)\Bigg]^2
     \end{split}
  \end{equation}
  where in the integration by parts the boundary term at $t=\infty$ is zero thanks to
  \eqref{eq:airyBd} and the assumption $g(t)\geq c-\kappa t^2$ with $\kappa\in(0,\frac34)$: in fact, for large $t$,
  \begin{align}\label{eq:just-ibp}
    &\pp_x(\tau_{g_a^-}\leq t)e^{-\frac23(\alpha+t)^3-(g^-_\alpha(t)-\lambda)(\alpha+t)-\beta x}\Ai(g^-_\alpha(t)-\lambda+(\alpha+t)^2)\\
    &~~\leq c\ts e^{-(\frac23-\kappa)(\alpha+t)^3-\frac23[(\alpha+t)^2-\kappa (\alpha+t)^2]^{3/2}+\mathcal{O}(\alpha+t)}
    \leq c\ts e^{-(\frac23-\kappa)(\alpha+t)^3-\frac23(1-\kappa)^{3/2}(\alpha+t)^3+\mathcal{O}(\alpha+t)}
  \end{align}
  for some $c>0$ (bounding the probability by 1), which goes to 0 as $t\to\infty$ for $\kappa<\frac34$.
  Using the Cauchy-Schwarz inequality we obtain
  \begin{multline}
    \|L_1\|_2^2\leq\int_{-\infty}^{g(\alpha)}dx\int_0^\infty\!\! dt\,\pp_x(\tau_{g_a^-}\leq t)^2e^{-2\beta x}e^{-2\beta't^2}\\
    \times\int_{-\infty}^0\!\!\!d\lambda\int_0^\infty\!\! dt\,
    \Big[e^{\beta't^2}\p_t\big(e^{-\frac23(\alpha+t)^3-(g^-_\alpha(t)-\lambda)(\alpha+t)}\Ai(g^-_\alpha(t)-\lambda+(\alpha+t)^2)\big)\Big]^2,\label{eq:C-S-pp}
  \end{multline}
  where $\beta'>0$ will be chosen later. 

  In order to estimate the first integral in \eqref{eq:C-S-pp} we introduce the line $r(s)=c-\kappa
  \alpha^2-\kappa(t+2\alpha)s$, which goes between $(0,c-\kappa \alpha^2)$ and
  $(t,c-\kappa(t+\alpha)^2)$. The piece of the integral with $x\in[c-\kappa
  \alpha^2,g(\alpha)]$ is clearly finite, so it remains to bound the piece with $x<c-\kappa \alpha^2$.
  Since $g_\alpha^-(s)\geq c-\kappa(s+\alpha)^2$, $r(s)$ lies below $g_\alpha^-(s)$ for $s\in[0,t]$ and
  thus we have, for $x<c-\kappa \alpha^2$, that
  \begin{align}
    \pp_x(\tau_{g_\alpha^-}\leq t)&=\pp_x(B(s)>g_\alpha^-(s)\text{ some }s\in[0,t])
    \leq\pp_x(B(s)>r(s)\text{ some }s\in[0,t])\\
    &=\int_0^t ds\,\tfrac{c-\kappa \alpha^2-x}{\sqrt{4\pi}s^{3/2}}e^{-\frac{(c-\kappa
        \alpha^2-x-\kappa(t+2\alpha)s)^2}{4s}}
  \end{align}
  where
  the last formula is standard (see e.g. (2.0.2) in Section 2 of \cite{handbookBM}).
  Using this we get
  \begin{align}
    &\int_{-\infty}^{c-\kappa \alpha^2}dx\int_0^\infty\!\! dt\,\pp_x(\tau_{g_a^-}\leq t)^2e^{-2\beta x}e^{-2\beta't^2}\\
    &\qquad\leq \int_0^\infty dt\int_{-\infty}^{c-\kappa \alpha^2}dx\left[\int_0^t ds\,\tfrac{c-\kappa
        \alpha^2-x}{\sqrt{4\pi t}s^{3/2}}e^{-\frac{(c-\kappa
          \alpha^2-x-\kappa(t+2\alpha)s)^2}{4s}}\right]^2e^{-2\beta x-2\beta't^2}\\
    &\qquad\leq\int_0^\infty dt\int_{-\infty}^{0}dx\left(\int_0^t ds\,s^{-2/3}\right)\left(
      \int_0^t ds\,\tfrac{s^{2/3}x^2}{4\pi ts^{3}}e^{-\frac{(x+\kappa(t+2\alpha)s)^2}{2s}}e^{-2\beta x-2\beta't^2-2\beta(c-\kappa \alpha^2)}\right),
  \end{align}
  where in the last line we used Cauchy-Schwarz again and shifted $x$ by $c-\kappa \alpha^2$.
  Computing the first $s$ integral and then changing variables $x\mapsto sx$ in the remaining multiple integral, this becomes
  \[3\int_0^\infty dt\,t^{1/3}\int_0^t ds\int_{-\infty}^{0}dx\,\tfrac{s^{2/3}x^2}{4\pi t}e^{-\frac{(x+\kappa(t+2\alpha))^2s}{2}}e^{-2\beta xs-2\beta't^2-2\beta(c-\kappa \alpha^2)},\]
  and then changing $s\mapsto ts$ we get
  \[3e^{-2\beta(c-\kappa \alpha^2)}\int_0^\infty dt\int_{-\infty}^{0}dx\int_0^1 ds\,\tfrac{ts^{2/3}x^2}{4\pi}e^{-\frac{(x+\kappa(t+2\alpha))^2}{2}st}e^{-2\beta xst-2\beta't^2}.\]
  If we extend the $x$ integral to all $\rr$ we obtain an upper bound, which is just a Gaussian integral in $x$ and can be computed explicitly, yielding that the above expression is bounded by
  \[c_1\ts e^{-2\beta(c-\kappa \alpha^2)}\int_0^\infty dt\int_0^1 ds\,\ts t^{-1/2}s^{-5/6}(1+st(2\beta+(t+2 \alpha)\kappa)^2)e^{-2(\beta'-\beta \kappa s)t^2+c_2 s t}\]
  for some $c_1,c_2>0$.
  This expression is clearly finite as long as we choose $\beta'>\kappa \beta$ (which of course we may), which
  shows that the first double integral in \eqref{eq:C-S-pp} is finite. 
  The other integral is simpler to estimate. Since $\Ai'$ satisfies the same estimate
  \eqref{eq:airyBd}, only with some minor polynomial corrections, we have (similarly to
  \eqref{eq:just-ibp}), for some $c>0$ (recall here $\lambda\leq0$),
  \begin{align}
    &\left[e^{\beta't^2}\p_t\big(e^{-\frac23(\alpha+t)^3-(g^-_\alpha(t)-\lambda)(\alpha+t)}\Ai(g^-_\alpha(t)-\lambda+(\alpha+t)^2)\big)\right]^2\\
    &\hspace{1.4in}\leq c\ts e^{2\lambda(t+\alpha)-(\frac43-2\kappa)(\alpha+t)^3-\frac43((1-\kappa)(\alpha+t)^2+4\alpha\kappa t+c+|\lambda|)^{3/2}+\mathcal{O}((\alpha+t)^2)}\\
    &\hspace{1.4in}\leq c\ts e^{2\lambda(t+\alpha)-\frac43|\lambda|^{3/2}}e^{-(\frac43-2\kappa)(\alpha+t)^3-\frac43(1-\kappa)^{3/2}(\alpha+t)^3+\mathcal{O}((\alpha+t)^2)}.
  \end{align}
  As before, since $-\frac43+2\kappa-\frac43(\lambda+(1-\kappa))^{3/2}<0$ for all $\kappa\in(0,\frac34)$,
  the last expression is integrable in $(\lambda,t)\in(-\infty,0]\times[0,\infty)$, and thus the second
  double integral in \eqref{eq:C-S-pp} is finite.  This finishes our proof that $L_1L_2$
  is finite. As we mentioned, the other two operators in \eqref{eq:KOmInfty} can be dealt
  with similarly.
\end{proof}

\begin{proof}[Proof of Theorem \ref{thm:limL}]
  As we mentioned, the proof of this result is essentially contained in the last proof. The idea is to rewrite $\overline{\scat}^{g}_L-\K$  in a similar way as \eqref{eq:KOmInfty}. In fact, the only difference is that the $t$ integrals appearing in the $\Phi^g_\alpha$ operators in \eqref{eq:KOmInfty} (see e.g. \eqref{eq:C-S-pp}) are now computed on $[0,\alpha+L]$ or $[0,L-\alpha]$ instead of $[0,\infty)$. Note however that the first term does not depend on $L$, just as the first term in \eqref{eq:KOmInfty}.
  This leaves us with estimating the difference between each of the last three terms in \eqref{eq:KOmInfty} with their corresponding ones in $\overline{\scat}^{g}_L$.
  Let us focus on the second of these three differences, which corresponds to the third term in \eqref{eq:KOmInfty}.
  It is given by $(A\bar P_0E_1)(E_2\bar P_0A^*)$ with
  \begin{align}
    E_1(x,y)&=\int_{-\infty}^0\!d\lambda\,e^{\alpha \Delta}A(x,\lambda)
    \left[\int_{L-\alpha}^\infty dt\,\pp_y(\tau_{g_\alpha^-}\in
      dt)e^{(\alpha-t)\Delta}A^*(\lambda,g_\alpha^-(t))\right]e^{-\beta x}\uno{y\leq g(\alpha)},\\
    E_2(y,x')&=\uno{y\leq g(\alpha)}e^{\beta x}\tsm\int_{-\infty}^0\!d\lambda\,(e^{\alpha\Delta}A^*)(\lambda,y)A^*(\lambda,x').
  \end{align}
  $E_2$ is the same as in the proof of Proposition \ref{prop:OmegaInfty-TrCl}, so the factor $E_2\bar P_0A^*$ is Hilbert-Schimdt (and it does not depend on $L$). Thus it sufices to prove that $\|E_1\|_2\longrightarrow0$ as $L\to\infty$.
  But following the exact same argument of the last proof we deduce (see \eqref{eq:C-S-pp}) that in this case
  \begin{multline}
    \|A\bar P_0E_1\|_2^2
    \leq\int_{-\infty}^{g(a)}dx\int_{L-\alpha}^\infty dt\,\pp_x(\tau_{g_\alpha^-}\leq t)^2e^{-2\beta x}e^{-2\beta't^2}\\
    \times\int_0^\infty d\lambda\int_{L-\alpha}^\infty
    dt\,\Big[e^{\beta't^2}\p_t\big(e^{-\frac23(t-\alpha)^3-(\lambda+g^-_\alpha(t))(t-\alpha))}\Ai(\lambda+g^-_\alpha(t)+(t-\alpha)^2)\big)\Big]^2.
  \end{multline}
  In the last proof we showed that both of these integrals are finite when the lower limit of integration in $t$ is replaced by 0, and hence they both go to 0 as $L\to\infty$. 
  This proves that $\overline{\scat}^{g}_L-\K \longrightarrow A\bar P_0\scat^{g}\bar P_0A^*-\K$.
  Since $\K$ is trace class, this finishes our proof.
\end{proof}

We finish this section with the proof that when $g\equiv r$,
\begin{equation}\label{eq:flatPsi}
  \Psi^{g^\pm_{\alpha}}_{\pm\alpha}=e^{\mp\alpha\xi}A^*-e^{\mp\alpha\xi}A^*\varrho_r,
\end{equation}
which was used in Example \ref{ex:GOEfromGral} (with $\alpha=0$) to show that in this case $\det\!\big(I-A\bar P_0\scat^{g}\bar P_0A^*)$ leads to the Tracy-Widom GOE distribution and in Example \ref{ex:2to1fromGral} to obtain similarly $F_{{\rm curved}\to{\rm flat}}$.
The key to the proof of \eqref{eq:flatPsi} is the following interesting identity:

\begin{prop}
Fix $y<0$ and let $\tau_0$ be the hitting time of the origin by a Brownian motion with diffusion coefficient 2 which starts at $y$.
Then for any $\alpha,\beta,\lambda\in\rr$,
\begin{equation}\label{eq:weirdId}
  \int_0^{\infty}\pp_y(\tau_0\in dt)e^{(\beta-t)\Delta}e^{\alpha\xi}A(0,\lambda)=e^{\beta\Delta}e^{\alpha\xi}A(-y,\lambda),
\end{equation}
or, explicitly, for any $x\in\rr$,
\begin{multline}\label{eq:weirdId2}
	\int_0^\infty dt\,\frac{|y|}{\sqrt{4\pi}t^{3/2}}e^{-y^2/4t}e^{2(\beta-t)^3/3+2\alpha(\beta-t)^2+\alpha^2(\beta-t)+(\beta-t)x}\Ai(x+(\beta-t)^2+2\alpha(\beta-t))\\
	=e^{2\beta^3/3+2\alpha\beta^2+\alpha^2\beta+\beta x-(\alpha+\beta)y}\Ai(x-y+\beta^2+2\alpha\beta).
\end{multline}
In particular, for all $y<0$ and $x\in\rr$,
\[\int_0^\infty dt\,\frac{|y|}{\sqrt{4\pi}t^{3/2}}e^{-y^2/4t}e^{-2t^3/3-tx}\Ai(x+t^2)=\Ai(x-y)\]
\end{prop}

The rough intuition behind \eqref{eq:weirdId}/\eqref{eq:weirdId2} is some sort of reflection principle, in which one runs a Brownian motion from $y$ until it hits the origin at some time $t$, and then runs it ``backwards in time'' from the origin for time $t$, leading to evaluation of $e^{\beta\Delta}e^{\alpha\xi}A(\cdot,\lambda)$ at $-y$ (the $\lambda$ in the above identity can be thought of as a fixed parameter).
In fact, and as will be clear from the proof, the identity should be true in greater generality, replacing $\Ai$ by any function $f$ which is in the preimage of the heat kernel and which satisfies suitable decay conditions.
Since we only need the case $f=\Ai$, for which the proof is quite simple, we do not pursue this any further.

\begin{proof}
  We assume for simplicity that $\alpha=\beta=0$, the general case being completely analogous (and requiring only slightly more complicated computations).
  We take $L>0$ and divide the $t$ integration into $(0,L)$ and $[L,\infty)$.
  Using the tail estimates on the Airy function \eqref{eq:airyBd} it is not hard to see that
  \[\int_L^\infty dt\,\frac{|y|}{\sqrt{4\pi}t^{3/2}}e^{-y^2/4t}e^{-2t^3/3-(x-y)t}\Ai(x-y+t^2)
  \xrightarrow[L\to\infty]{}0,\]
  where we are using the explicit representation of our integral.
  Hence we need to show that
  \begin{equation}\label{eq:LtoinftyFlat}
    \int_0^{L}\pp_y(\tau_0\in dt)e^{-t\Delta}A(0,\lambda)\xrightarrow[L\to\infty]{}A(-y,\lambda).
  \end{equation}
  We may rewrite the left hand side as
  \[\int_{-\infty}^\infty dz\int_0^{L}\pp_y(\tau_0\in dt)e^{(L-t)\Delta}(0,z)e^{-L\Delta}A(z,\lambda),\]
  where we have used Fubini.
  Note that the factor $\int_0^{L}\pp_y(\tau_0\in dt)e^{(L-t)\Delta}(0,z)$ is nothing but $\pp_y(B_t\geq0\text{ for some }t\leq L,\,B_L\in dz)/dz$.
  For positive $z$ this is just the transition density of Brownian motion (recall $y<0$), while for $z<0$ it can be computed explicitly by the reflection principle and is given by $(4\pi L)^{-1/2}e^{-(y+z)^2/4L}$.
  Hence the last expression equals, setting $x=-\lambda$
  \begin{multline}
    \int_0^{\infty}dz\tfrac{1}{\sqrt{4\pi L}}e^{-(y-z)^2/4L}e^{-2L^3/3-L(x+z)}\Ai(x+z+L^2)\\
    +\int_{-\infty}^0dz\tfrac{1}{\sqrt{4\pi L}}e^{-(y+z)^2/4L}e^{-2L^3/3-L(x+z)}\Ai(x+z+L^2)
  \end{multline}
  Using the contour integral representation of the Airy function 
  \[\Ai(x)=\frac1{2\pi\I}\int_{\gamma}du\,e^{u^3/3-xu},\]
  where we take the contour $\gamma$ to be composed of the two rays starting the origin with angles $\pm\theta$ for some $\theta\in(\pi/6,\pi/4)$, and Fubini to compute the $z$ integrals, the above equals (here $\delta>0$)
  \begin{multline}
    \frac{1}{2\pi\I}\int_{\gamma}du\,\tfrac12e^{(u+L)^3/3-(x+y)(u+L)}
    {\rm Erfc}(-\tfrac1{2\sqrt{L}}y+\sqrt{L}(u+L))\\
    +\frac{1}{2\pi\I}\int_{\gamma}du\,\tfrac12e^{(u+L)^3/3-(x-y)(u+L)}
    {\rm Erfc}(-\tfrac1{2\sqrt{L}}y-\sqrt{L}(u+L)),
  \end{multline}
  where Erfc is the complementary error function ${\rm Erfc}(z)=2/\sqrt{\pi}\int_{x}^\infty ds\,e^{-s^2}$.
  After changing $u$ to $u-L$ and shifting the contour $L+\gamma$ back to $\gamma$, this becomes
   \begin{multline}
    \frac{1}{2\pi\I}\int_{\gamma}du\,\tfrac12e^{u^3/3-(x-y)u+(x-y)\alpha}
    {\rm Erfc}(-\tfrac1{2\sqrt{L}}y-\sqrt{L}u)\\
    +\frac{1}{2\pi\I}\int_{\gamma}du\,\tfrac12e^{u^3/3-(x+y)u+(x+y)\alpha}
    {\rm Erfc}(-\tfrac1{2\sqrt{L}}y+\sqrt{L}u).
  \end{multline}
  As can be seen from the decay of the factor $e^{u^3/3}$ on the given contour, the $L\to\infty$ limit can be taken inside both integrals, and then by our choice of $\gamma$, and since ${\rm Erfc}(z)$ goes to 0 as $z\to\infty$ inside the sector $|\!\arg(z)|<\pi/4-\delta$ and to 2 as $z\to\infty$ inside the sector $3\pi/4+\delta<\arg(z)<5\pi/4-\delta$ (see \cite[(7.2.4)]{NIST:DLMF}), we see that the above expression converges as $L\to\infty$ to
  $\frac{1}{2\pi\I}\int_{\delta+\I\rr}du\,e^{u^3/3-(x-y)u}=\Ai(x-y)$
  as needed.
\end{proof}

Now we can prove \eqref{eq:flatPsi}, which we do only in the case involving $g^+$ (the other being completely analogous).
Recall that we are considering $g\equiv r$.
The probability $\pp_y(\tau_{g^+_\alpha}\in dt)$ in the integral appearing in the definition of $\Psi^{g^+_\alpha}_\alpha$ equals $\pp_{y-r}(\tau_0\in dt)$ with $\tau_0$ as in the proposition.
Hence, since $A^*e^{-t\Delta}e^{-\alpha\xi}A^*(\lambda,r)=e^{-t\Delta}e^{-\alpha\xi}A(0,\lambda-r)$, the integral equals
\[\int_0^\infty\pp_{y-r}(\tau_0\in dt)e^{-t\Delta}e^{-\alpha\xi}A(0,\lambda-r)=e^{-\alpha\xi}A(r-y,\lambda-r)=e^{-\alpha\xi}A^*\varrho_r(\lambda,y).\]
This means that $\Psi^{g^+}_\alpha(\lambda,y)=e^{-\alpha\xi}A^*(\lambda,y)-e^{-\alpha\xi}A^*\varrho_r(\lambda,y)$, as desired.

\appendix 

\section{Baik-Rains beats Tracy-Widom}

\subsection{Proof of the inequality}\label{sec:BRbeatsTW}

The purpose of this section is to prove the following

\begin{thm}\label{thm:BRTW}
  The Baik-Rains distribution with $\gamma=0$ dominates stochastically the Tracy-Widom GOE distribution:
  \[F_{\rm stat}^0(r)<F_{\rm GOE}(r)\]
  for all $r\in\rr$.
\end{thm}

The proof is based on the representations of $F_{\rm GOE}$ and $F_{\rm stat}^0$ in terms Painlev\'e transcendents.
We follow here the presentation in \cite{baikRainsF0}.

Let $u(x)$ be the Hastings-McLeod solution to the Painlev\'e II equation
\begin{equation}
  u_{xx}=2u^3+xu,\label{eq:PII}
\end{equation}
singled out as the unique solution satisfying the boundary condition
\begin{equation}\label{eq:bdHMcL}
  u(x)\sim -\Ai(x)\quad\text{as}\quad x\to\infty.
\end{equation}
The existence and uniqueness of this solution were established in \cite{hastingsMcLeod}, where it is also shown that
\begin{equation}\label{eq:uProps}
  u(x)<0\qqand u'(x)>0\qquad\text{ for all }x\in\rr.
\end{equation}
The asymptotics as $x\to -\infty$ are given by
\begin{equation}\label{eq:asHM}
  u(x) = -\sqrt{\frac{-x}{2}}\biggl( 1+\mathcal{O}\bigl(\tfrac1{x^2}\bigr)\biggr),
\qquad\text{as $x\to -\infty$,}
\end{equation}
see e.g. \cite{hastingsMcLeod,deiftZhouPainleveII}.
Define
\begin{equation}\label{eq:def-v}
  v(x)=-\int_x^{\infty} (u(s))^2 ds=u(x)^4+xu(x)^2-(u'(x))^2,
\end{equation}
so that $v'(x)= (u(x))^2$ (the second equality in \eqref{eq:def-v} is \cite[2.6]{baikRainsF0}, and is a consequence of \eqref{eq:PII}).
Also define
\begin{equation}\label{eq:def-y}
  y(x)=x+2u'(x)+2u(x)^2,
\end{equation}
together with
\begin{equation}
F(x) = e^{\frac12\int_x^{\infty}ds\ts v(s)}
= e^{-\frac12\int_x^{\infty}ds\ts (s-x)(u(s))^2} \qqand
E(x) = e^{\frac12\int_x^{\infty}ds\ts u(s)}.
\end{equation}

The Tracy-Widom GUE and GOE distributions and the Baik-Rains distribution (with $\kappa=0$) can be expressed in terms of $u$ and $v$ as follows:
\begin{align}  
\label{eq:GUE-Painl}
   F_{\rm GUE}(x) &= F(x)^2,\\
\label{eq:GOE-Painl}
   F_{\rm GOE}(x) &= F(x)E(x),\\
\label{eq:BR-Painl}
   F_{\rm stat}^0(x) &= [1-y(x)v(x)]E(x)^4F_{\rm GUE}(x)
\end{align}
(the first two were derived by Tracy and Widom in \cite{tracyWidom,tracyWidom2} while the last one amounts to the original definition of $F_{\rm stat}^0$, as given in \cite{baikRainsF0}).
In view of these expressions, the proof of Theorem \ref{thm:BRTW} amounts to showing that
\begin{equation}\label{eq:BRineq}
  [1-y(x)v(x)]E(x)^4F(x)^2<E(x)F(x).
\end{equation}
Since both $F(x)$ and $E(x)$ are in $(0,1)$ for all $x$, this follows from the following stronger inequality:
\begin{equation}\label{eq:BRbTWstr}
  (1-y(x)v(x))E(x)\leq1.
\end{equation}

The main step of the proof of this inequality consists in showing that $y$ is increasing.

\begin{lem}\label{lem:yIncr}
$y(x)>0$ and $y'(x)>0$ for all $x$.
\end{lem}

\begin{proof}
  We note first that $y$ satisfies
 \begin{subequations}
    \begin{align}
       y'(x)&=1+2u(x)y(x),\label{eq:yProps1}\\
       y(x)&=1/\sqrt{-2x}\big(1+\mathcal{O}(1/x^2)\big)~~\text{ as $x\to-\infty$}\label{eq:yProps2}
   \end{align}
 \end{subequations}
  (this is \cite[(2.19)]{baikRainsF0}, but it follows easily from the definition of $y$ and the asymptotics \eqref{eq:asHM}).
  
  Define $b=\sup\{x_0\in\rr\!:\,y(x)\geq0~\forall\,x\leq x_0\}$.
  Note that $b>- \infty$ thanks to \eqref{eq:yProps2}.
  We claim that in fact $b= \infty$.
  To this end we differentiate \eqref{eq:yProps1} to obtain
  \[y''(x)=2u'(x)y(x)+2u(x)y'(x).\]
  Since $y(x)\geq0$ for all $x\leq b$ and $u$ is increasing by \eqref{eq:uProps}, and letting $z(x)=-y'(x)$, this identity gives $z'(x)\leq2u(x)z(x)$ for all $x\in[a,b]$, where $a<b$ is arbitrary.
  Now using Gronwall's inequality we get $z(x)\leq2z(a)\exp(\int_a^x dt\,u(t))$ or, in other words,
  \begin{equation} \label{eq:gronw} 
    y'(x)\geq2y'(a)e^{\int_a^x dt\,u(t)},
  \end{equation}
  for all $x\in[a,b]$.
  On the other hand, using \eqref{eq:asHM}, \eqref{eq:def-y} and \eqref{eq:yProps2} to solve for the asymptotics for $u'$ leads to
  \begin{equation}\label{eq:as-uprime}
     u'(x)=\tfrac1{2\sqrt{-2x}}+\mathcal{O}\big(\tfrac1{\sqrt{-x}}\big)\qquad\text{as }x\to- \infty,
   \end{equation}
  and thus using \eqref{eq:def-y} in \eqref{eq:yProps1} together \eqref{eq:as-uprime} and  \eqref{eq:asHM} again we deduce that $y'(x) \longrightarrow0$  as $x\to- \infty$.
  Therefore, taking $a\to-\infty$ in \eqref{eq:gronw} we deduce that $y'(x)\geq0$ for all $x\leq b$ (note that $u<0$).
  Since $u$ is not a linear function (on any subinterval of $(- \infty,b]$) and is strictly increasing, \eqref{eq:def-y} implies that $y$ is not constant on any subinterval of $(- \infty,b]$, and thus we can choose an $a_0<b$ with arbitrarily large absolute value such that $y'(a_0)>0$.
  As a consequence, \eqref{eq:gronw} again gives 
  \begin{equation}\label{eq:yPrimePos}
    y'(x)>0\qquad\text{ for all }x\leq b.
  \end{equation}
  Now suppose $b < \infty$.
  By continuity we have $y(b)=0$, but the fact that $y'(b)>0$ implies that there is a $b_1>b$ such that $y(x)>0$ on $(b,b_1)$, and this contradicts the choice of $b$.

  As a consequence, we have $y(x)>0$ for all $x$, which is the first statement of the lemma.
  But then the same argument leads to \eqref{eq:yPrimePos} being valid for all $x\in\rr$, which completes the proof.
\end{proof}

\begin{proof}[Proof of \eqref{eq:BRbTWstr}]
  Since $v'(x)=u(x)^2$, we have
   \begin{align}
    \Big([1-yv]E\Big)'&=-[y'v+yv'+\tfrac12u-\tfrac12uyv]E= -[y'v+yu^2+\tfrac12u-\tfrac12uyv]E\\
    &=-[y'v+(\tfrac12+yu)u]E+\tfrac12uyvE=-y'(v+\tfrac12u)E+\tfrac12uyvE>0,
  \end{align} 
  where we have used Lemma \ref{lem:yIncr} and the fact that $u$ and $v$ are strictly negative and $E$ is strictly positive.
  On the other hand, by \eqref{eq:airyBd} and \eqref{eq:def-v} we have $|v(x)|\sim C\ts e^{-\frac43x^{3/2}}$ for large enough $x$, and thus it is clear from \eqref{eq:bdHMcL}, \eqref{eq:def-y} and \eqref{eq:airyBd} again that $y(x)v(x) \longrightarrow0$ as $x\to \infty$.
  This implies that $[1-y(x)v(x)]E(x)\longrightarrow1$ as $x\to \infty$ which, together with the fact that $[1-y(x)v(x)]E(x)$ is increasing, yields the result.
\end{proof}

\subsection{An inequality for last passage percolation with boundary conditions}

Observe that Lemma \ref{lem:yIncr} yields $1-y(x)v(x)>1$, so $F^0_{\rm stat}(x)>E(x)^4F_{\rm GUE}(x)$ by \eqref{eq:BR-Painl}.
But then using \eqref{eq:GUE-Painl} and \eqref{eq:GOE-Painl} in this inequality yields
\begin{equation}\label{eq:extraOrd-App}
  F_{\rm GOE}(r)^4<F_{\rm stat}^0(r)F_{\rm GUE}(r),
\end{equation}
for all $r$, which we had already stated in the introduction as \eqref{eq:extraOrd}.

This inequality can be understood as a statement about a certain \emph{last passage percolation model}.
Consider a family $\big\{w(i,j)\}_{i,j\in\nn}$ of \emph{weights}, defined as independent geometric random variables with parameters $q_{i,j}$ (i.e. $\pp(w(i,j)=k)=(1-q_{i,j})q_{i,j}^{k}$ for $k\geq0$) given as
\begin{equation}\label{eq:lppParams}
  q_{i,j}=%
  \begin{dcases*}
    q & if $i,j>1$\\
    \alpha_+\sqrt{q} & if $i>1$, $j=0$\\
    \alpha_-\sqrt{q} & if $i=0$, $j>1$\\
    0 & if $i=j=0$
  \end{dcases*}
\end{equation}
for some constants $q,\alpha_,\alpha_+\in[0,1]$ with $q>0$.
The \emph{point-to-point last passage time} is defined, for $N\in\nn$, as
\[L_{\alpha_-,\alpha_+}(N)=\max_{\pi:(0,0)\to(N,N)}\sum_{i=0}^{2N}w(\pi_i)\]
(we omit $q$ from the notation for simplicity) where the maximum is taken over all up-right paths (i.e. such that $\pi_i-\pi_{i-1}\in\{(1,0),(0,1)\}$ connecting the origin to $(N,N)$.
\citet{baikRainsF0} (see Section 4) showed that there are explicit constants $c_1$ and $c_2$, depending only on $q$, such that 
\begin{multline}\label{eq:johBR}
  \lim_{N\to \infty}\pp\bigg(\frac{L_{\alpha_-,\alpha_+}(N)-c_1N}{c_2N^{1/3}}\leq r\bigg)\\
  =\begin{dcases*}
     F_{\rm GUE}(r) & if $0\leq\alpha_-,\alpha_+<1$,\\
     F_{\rm GOE}(r)^2 & if $0\leq\alpha_-<1$ and $\alpha_+=1$ or $0\leq\alpha_+<1$ and $\alpha_-=1$,\\
     F^0_{\rm stat}(r) & if $\alpha_-=\alpha_+=1$.
   \end{dcases*}
\end{multline}

Now consider two independent copies of the model with the same $q$ and denote the last passage times associated to them as $L^1_{\alpha^1_-,\alpha^1_+}(N)$ and $L^2_{\alpha^2_-,\alpha^2_+}(N)$.
Then it is reasonable to guess that
\begin{equation}\label{eq:lppOrd}
  \max\{L^1_{1,0}(N),L^2_{0,1}(N)\}\succ\max\{L^1_{0,0}(N),L^2_{1,1}(N)\}
\end{equation}
for all $N$, where $\succ$ denotes stochastic domination, since on the left hand side both last passage times get to use boundary weights in their choice of optimal paths. while on the right hand side the first last passage time has no available boundary weights, while the second one has to choose whether to use the ones on the vertical or on the horizontal axes (if any).
In view of \eqref{eq:johBR}, \eqref{eq:extraOrd-App} is just the limiting version of this inequality.
It would be interesting to turn this intuition into a proof of \eqref{eq:lppOrd}, or else to show that the inequality is false (in which case only \eqref{eq:extraOrd-App} only becomes true in the limit).

\vspace{14pt}

\noindent{\bf Acknowledgements.}
JQ gratefully acknowledges financial support from the Natural Sciences and Engineering
Research Council of Canada, the I.~W. Killam Foundation, and the Institute for Advanced
Study. DR was partially supported by Fondecyt Grant 1160174, by Conicyt Basal-CMM, and by
Programa Iniciativa Cient\'ifica Milenio grant number NC130062 through Nucleus Millenium
Stochastic Models of Complex and Disordered Systems; he also thanks the Institute for Advanced Study for its hospitality during a visit at which this project got started.

\printbibliography[heading=apa]

\end{document}